\documentclass[11pt,a4paper]{article}

\usepackage{header}
\usepackage{abbrev}

\begin{document}
\begin{sloppypar}

\title{Toroidal area-preserving parameterizations of genus-one closed surfaces}

\author{\MakeUppercase{Marco Sutti}\thanks{Division of Mathematics, Gran Sasso Science Institute, L'Aquila, Italy (\email{marco.sutti@gssi.it}).}\hspace{2mm}\orcidlink{0000-0002-8410-1372} \MakeUppercase{and Mei-Heng Yueh}\thanks{Department of Mathematics, National Taiwan Normal University, Taipei, Taiwan (\email{yue@ntnu.edu.tw}).}\hspace{2mm}\orcidlink{0000-0002-6873-5818}}

\date{\today}

\maketitle

\begin{abstract}

We consider the problem of computing toroidal area-preserving parameterizations of genus-one closed surfaces. We propose four algorithms based on Riemannian geometry: the projected gradient descent method, the projected conjugate gradient method, the Riemannian gradient method, and the Riemannian conjugate gradient method.
Our objective function is based on the stretch energy functional, and the minimization is constrained on a power manifold of ring tori embedded in three-dimensional Euclidean space. Numerical experiments on several mesh models demonstrate the effectiveness of the proposed framework. Finally, we show how to use the proposed algorithms in the context of surface registration and texture mapping applications.



\end{abstract}

\bigskip

\section{Introduction}

In recent years, parameterizations of manifolds have found many applications in computer graphics and medical imaging. While many efficient methods have been developed for computing angle-preserving (i.e., conformal) mappings, computing area-preserving mappings (also called authalic or equiareal) of closed surfaces with nontrivial topology is a topic that has received less attention. 

This work focuses on the computation of toroidal area-preserving parameterizations of genus-one closed surfaces. The focus of our study, the ring torus, is illustrated in Figure~\ref{fig:torus_and_tangent_plane}.

\begin{figure}[htbp]
  \centering
  \includegraphics[width=0.55\textwidth]{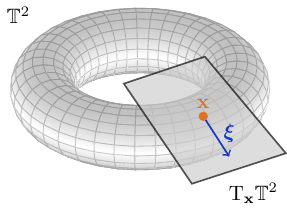}
\caption{A ring torus with a tangent plane at a point $\mathbf{x}$.}
\label{fig:torus_and_tangent_plane}
\end{figure}

Our approach is based on the stretch energy minimization (SEM)~\protect{\cite{Yueh:2020,Yueh:2023}}, which has also been used in the recent work \protect{\cite{Sutti_Yueh:2024}} to compute spherical area-preserving mappings of genus-zero closed surfaces.
Here, the minimization is performed on the power manifold of $n$ ring tori.
The initial torus mapping is computed by minimizing the stretch energy of the mapping in variables of the planar fundamental domain using the fixed-point method \protect{\cite{Yueh:2020}}, which is a modification of the holomorphic differential method introduced in~\protect{\cite{GuYa02}}. During the last fifteen years, many numerical algorithms based on the minimization of the area distortion have been developed to find the area-preserving parameterizations of closed surfaces to a sphere $\Stwo$ or a disk $\mathbb{B}^2$ for a 2-manifold of genus zero \protect{\cite{Zhao:2013,Su:2016,YLWY:2019,CR:2018,Sutti_Yueh:2024}}. 

In contrast, toroidal surfaces remain relatively unexplored. Dey et al.~\protect{\cite{Dey:2013}} proposed an efficient algorithm called \texttt{ReebHanTun} to compute a basis for handle and tunnel loops using the concept of the Reeb graph~\protect{\cite{Dey:2008}}, which provides an initial set of loops that constitute a handle and tunnel basis. Yueh et al.~\protect{\cite{Yueh:2020}} proposed computing a volume-preserving parameterization of genus-one 3-manifolds and area-preserving maps of their boundary. The more recent work by Yao and Choi \protect{\cite{Yao:2026}} also considers toroidal maps for genus-one surfaces but uses a different approach based on density-equalizing.

To the authors' knowledge, this is the first work that targets toroidal area-preserving maps for genus-one closed surfaces using projected and Riemannian optimization methods.

\subsection{Contributions} 
The main contributions of the present work are the following:
\begin{enumerate}
    \item We develop the geometry of a ring torus needed to generalize the existing algorithms to toroidal surfaces.
    \item The Riemannian optimization algorithms minimize the stretch energy to compute the area-preserving mappings between genus-one closed surfaces and a ring torus $\Torus$.
    \item Numerical experiments show the effectiveness and robustness of the proposed algorithms, showing that conjugate gradient algorithms provide better results.
    \item We show how to use the proposed algorithms for applications involving vertebrae registration and texture mappings.
\end{enumerate}

\subsection{Outline of the paper}

The rest of the paper is organized as follows. Section~\ref{sec:simplicial_objective} introduces the main concepts on simplicial surfaces and mappings, and presents the formulation of the objective function. Sect.~\ref{sec:fundamental_domain} gives some background on how to compute the fundamental domain.
Sect.~\ref{sec:riem_optim_and_geometry} briefly introduces the Riemannian optimization framework and explains the geometry of the ring torus and the tools needed to perform optimization on the power manifold of $n$ ring tori.
Sect.~\ref{sec:riem_optim_and_geometry} describes the proposed algorithms. Sect.~\ref{sec:numerical_experiments} discusses the numerical experiments carried out to compare and evaluate our algorithms in terms of accuracy and efficiency. Sect.~\ref{sec:applications} provides concrete applications for the surface registration of two vertebrae and texture mapping.
Finally, we wrap our paper with concluding remarks and future outlook in Sect.~\ref{sec:conclusions}.
Appendix~\ref{sec:line_search} gives more details about the line-search procedure used in our methods.

\subsection{Notation}\label{sec:notation}

In this section, we list the paper's notations and symbols adopted in order of appearance in the paper. Symbols specific to a particular section are usually not included in this list.

\begin{table}[htbp]
   \begin{center}
      \begin{tabular}{ll}
          $ \tau $        &  Triangular face \\
          $ \left\vert \tau \right\vert $  &  Area of the triangle $ \tau $ \\
          $ \cM $         &  Simplicial surface \\
          $ \mathcal{V}(\cM) $ & Set of vertices of $\cM$ \\
          $ \mathcal{F}(\cM) $ & Set of faces of $\cM$ \\
          $ \mathcal{E}(\cM) $ & Set of edges of $\cM$ \\
          $ \bv_{i}, \ \bv_{j}, \ \bv_{k} $      &  Vertices of a triangular face \\
          $ f $           &  Simplicial mapping \\
          $ \f $          &  Representative matrix of $f$ \\
          $ \f_{\ell} $   &  Coordinates of a vertex $ f(\bv_{\ell}) $ \\
          $ \vecop $      &  Column-stacking vectorization operator \\
          $ \Torus $      &  Ring torus in $ \R^{3} $ \\
          $ \powerTorus $ &  Power manifold of $n$ tori in $ \R^{3} $ \\
          $ E_{A}(f) $    &  Authalic energy \\
          $ E_{S}(f) $    &  Stretch energy \\
          $ L_{S}(f) $    &  Weighted Laplacian matrix \\
          $ w_{S} $  &  Modified cotangent weights \\
          $ \cA(f) $      &  Area of the image of $f$ \\
          $ \mathrm{T}_{\x}\Torus $    &  Tangent space to $\Torus$ at $\x$ \\
          $ \Pi_{\Torus} $         & Projection of a point onto $\Torus$ \\
          $ \P_{\x} $    & Orthogonal projector onto the tangent space to $\Torus$ at $\x$ \\
          $ \P_{\mathrm{T}_{\f_{\ell}}\powerTorus} $  & Orthogonal projector onto the tangent space to $\powerTorus$ at $\f_{\ell}$ \\
          $\Retraction$  &  Retraction mapping \\
          $\nabla E(f)$  &  Euclidean gradient of $E(f)$ \\
          $\grad E(f)$   &  Riemannian gradient of $E(f)$
      \end{tabular}
   \end{center}
\end{table}

\section{Simplicial surfaces and mappings, and objective function} \label{sec:simplicial_objective}

To provide some basic background about the objects that we want to optimize, we briefly introduce the simplicial surfaces and mappings in Sect.~\ref{sec:simplicial_surfaces}, and then our objective function in Sect.~\ref{sec:objective_function}.

\subsection{Simplicial surfaces and mappings} \label{sec:simplicial_surfaces}

A simplicial surface parameterization is a bijective mapping between the simplicial surface and a domain with a simple canonical shape.
Formally, a simplicial surface $\cM$ is the underlying set of a simplicial $2$-complex $\mathcal{K}(\cM) = \cF(\cM)\cup\cE(\cM)\cup\mathcal{V}(\cM)$ composed of vertices
$$
\mathcal{V}(\cM) = \left\{\bv_\ell = \left( v_\ell^1, v_\ell^2, v_\ell^2 \right)\tr \in\R^3 \right\}_{\ell=1}^n,
$$
oriented triangular faces
$$
\cF(\cM) = \left\{ \tau_\ell = [\bv_{i_\ell}, \bv_{j_\ell}, \bv_{k_\ell}] \mid \bv_{i_\ell}, \bv_{j_\ell}, \bv_{k_\ell} \in\mathcal{V}(\cM) \right\}_{\ell=1}^m,
$$
and directed edges
$$
\cE(\cM) = \left\{ [\bv_i,\bv_j] \mid [\bv_i,\bv_j,\bv_k]\in\cF(\cM) \text{ for some $\bv_k\in\mathcal{V}(\cM)$} \right\}.
$$
A simplicial mapping $f\colon \cM\to\R^3$ is a particular type of piecewise affine mapping with the restriction mapping $f|_\tau$ being affine, for every $\tau\in\cF(\cM)$.
We denote 
$$
\f_\ell \coloneqq f(\bv_\ell) = \left( f_\ell^1, f_\ell^2, f_\ell^3 \right)\tr, ~ \textrm{ for every $\bv_\ell\in\mathcal{V}(\cM)$}.
$$
The mapping $f$ can be represented as a matrix
\begin{equation}\label{eq:representative_matrix}
    \f 
= \begin{bmatrix}
\f_1\tr \\
\vdots \\
\f_n\tr
\end{bmatrix}
= \begin{bmatrix}
f_1^1 & f_1^2 & f_1^3 \\
\vdots & \vdots & \vdots \\
f_n^1 & f_n^2 & f_n^3
\end{bmatrix}
\eqqcolon \begin{bmatrix}
\f^1 & \f^2 & \f^3
\end{bmatrix},
\end{equation}
or a vector
$$
\vecop(\f) = \begin{bmatrix}
\f^1 \\
\f^2 \\
\f^3
\end{bmatrix}.
$$

\subsection{The objective function}\label{sec:objective_function}

The authalic or equiareal energy for simplicial mappings $f\colon \cM\to\R^{3}$ is defined as
\[
   E_A(\f) = E_S(\f) - \cA(\f),
\]
where $E_{S}$ is the stretch energy defined as
$$
   E_{S}(\f) = \frac{1}{2} \, \vecop(\f)\tr (I_3\otimes L_S(\f)) \vecop(\f).
$$
Here, $ I_{3} $ is the identity matrix of size 3-by-3, $\otimes$ denotes the Kronecker product, and $ L_S(\f) $ is the weighted Laplacian matrix defined by
\begin{equation} \label{eq:L_S}
   [L_S(\f)]_{i,j} =
\begin{cases}
-\sum_{[\bv_i,\bv_j,\bv_k]\in\cF(\cM)} [w_S(\f)]_{i,j,k}  &\mbox{if $[{\bv}_i,{\bv}_j]\in\cE(\cM)$,}\\
-\sum_{\ell\neq i} [L_S(\f)]_{i,\ell} &\mbox{if $j = i$,}\\
0 &\mbox{otherwise.}
\end{cases}
\end{equation}
The modified cotangent weights $w_S(\f)$ are defined as
\begin{equation} \label{eq:omega}
   [w_S(\f)]_{i,j,k} = \frac{\cot(\theta_{i,j}(\f))  \, |[\f_i,\f_j,\f_k]|}{2|[\bv_i,\bv_j,\bv_k]|},
\end{equation}
with $\theta_{i,j}(\f)$ being the angle opposite to the edge $[\f_i,\f_j]$ at the point $\f_k$ on the image $f(\cM)$, as illustrated in Figure \ref{fig:cot}.

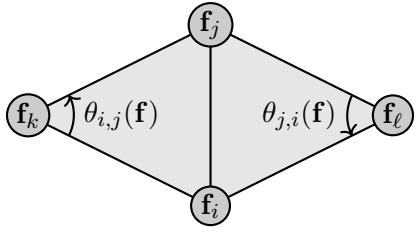
\begin{figure}[htbp]
\centering
\begin{tikzpicture}[thick,scale=1.2]
\coordinate (v_i) at (0,0);
\coordinate (v_j) at (0,2);
\coordinate (v_k) at (2,1);
\coordinate (v_l) at (-2,1);
\filldraw[black!10] (v_i) -- (v_j) -- (v_k);
\filldraw[black!10] (v_i) -- (v_j) -- (v_l);
\pic[draw, ->, "$\theta_{i,j}(\f)$", angle eccentricity=2.05, angle radius=0.6cm]{angle = v_i--v_l--v_j};
\pic[draw, ->, "$\theta_{j,i}(\f)$", angle eccentricity=2.05, angle radius=0.6cm]{angle = v_j--v_k--v_i};
\draw{
(v_i) -- (v_j) -- (v_k) -- (v_i) -- (v_l) -- (v_j)
};
\tikzstyle{every node}=[circle, draw, fill=black!20, inner sep=1pt, minimum width=2pt]
\draw{
(0,0) node{$\f_i$}
(0,2) node{$\f_j$}
(2,1) node{$\f_\ell$}
(-2,1) node{$\f_k$}
};
\end{tikzpicture}
\caption{An illustration of the cotangent weight defined on the image of $f$.}
\label{fig:cot}
\end{figure}

It is proved in~\protect{\cite[Corollary 3.4]{Yueh:2023}} that $E_A(\f)\geq 0$ and the equality holds if and only if $f$ preserves the area.

In this work, we consider as objective function the following formulation with a prefactor
\begin{equation}\label{eq:objective_function}
   E(\f) = \frac{|\cM|}{\cA(\f)} E_{S}(\f) - \cA(\f).
\end{equation}
The prefactor $|\cM|/\cA(\f)$ is because, due to the optimization process, the image area $\cA(\f)$ is not constant. This objective function has a property analogous to that of $E_A(\f)$ in~\protect{\cite[Corollary 3.4]{Yueh:2023}}, which is stated in the following theorem. 

\begin{theorem}[{\cite[Theorem 1]{LiuYueh:2024}}]
The objective function \eqref{eq:objective_function} satisfies $E(\f)\geq 0$, and the equality holds if and only if $f$ is area-preserving.
\end{theorem}
\begin{proof}
By applying the Cauchy--Schwarz inequality on the sequences 
$\left\{\sqrt{|\tau|}\right\}_{\tau\in\cF(\cM)}$ and $\left\{|f(\tau)|/\sqrt{|\tau|}\right\}_{\tau\in\cF(\cM)}$ implies 
$$
\left(\sum_{\tau\in\cF(\cM)} \left( \sqrt{|\tau|} \right)^2 \right)
\left(\sum_{\tau\in\cF(\cM)} \left(\frac{|f(\tau)|}{\sqrt{|\tau|}} \right)^2 \right)
\geq 
\left(\sum_{\tau\in\cF(\cM)}|f(\tau)|\right)^2.
$$
In other words,
$$
|\cM|\,E_S(\f) \geq \cA(\f)^2.
$$
Noting that $\cA(\f)>0$, dividing by $\cA(\f)$ gives
\[
E(\f) = \frac{|\cM|}{\cA(\f)} E_S(\f) - \cA(\f)\geq 0.
\]
Moreover, the equality holds precisely when $\frac{|f(\tau)|}{\sqrt{|\tau|}}$ is proportional to $\sqrt{|\tau|}$, i.e., $\frac{|f(\tau)|}{|\tau|}$ is constant. Hence, $E(\f)=0$ if and only if $f$ scales each face by the same factor, i.e., $f$ is area-preserving.
\end{proof}

To perform numerical optimization via the proposed methods, we need to compute the (Euclidean) gradient, which is given by the following proposition.

\begin{proposition}[Formula for $\nabla E$]
The gradient of $E(\f)$ can be explicitly formulated as
\begin{equation} \label{eq:grad_E}
   \nabla E(\f) = \frac{2|\cM|}{\cA(\f)} \, L_S(\f) \, \f - \left( 1 + \frac{|\cM| E_S(\f)}{\cA(\f)^2} \right) \nabla\cA(\f).
\end{equation}
\end{proposition}

\begin{proof}
The Leibniz rules indicate 
$$
\nabla E(\f) = \frac{|\cM|}{\cA(\f)} \nabla E_S(\f) + E_S(\f) \nabla \frac{|\cM|}{\cA(\f)} - \nabla\cA(\f).
$$
The desired \eqref{eq:grad_E} is obtained by applying the formula $\nabla E_S(\f) = 2\, L_S(\f) \, \f$ from \cite[(3.6)]{Yueh:2023}.
\end{proof}

Here, $\nabla E(\f)$ is an $n$-by-3 matrix obtained by reshaping the gradient vector of length $3n$. The term $\nabla\cA(\f)$ in \eqref{eq:grad_E} is explicitly formulated in the following proposition.

\begin{proposition}[Formula for $\nabla\cA$]
On a triangle $\tau=[\bv_i, \bv_j, \bv_k]\in\cF(\cM)$, the gradient of $\cA$ can be explicitly formulated as
\begin{equation} \label{eq:GradA}
\nabla\cA(\f_\tau) = \frac{|\tau|}{\cA(\f_\tau)} \, L_S(\f_\tau) \,\f_\tau,
\end{equation}
where $\f_\tau=[\f_i, \f_j, \f_k]^\top$.
\end{proposition}
\begin{proof}
Using the explicit formulas
$$
E_S(\f_\tau) = \frac{\cA(\f_\tau)^2}{|\tau|}, 
\quad
\nabla E_S(\f_\tau) = 2\,L_S(\f_\tau)\,\f_\tau
\quad\text{(see~\cite[Lemma 3.1 and Theorem 3.5]{Yueh:2023})},
$$
the chain rule gives
$$
\nabla E_S(\f_\tau) 
= \nabla \left(\frac{\cA(\f_\tau)^2}{|\tau|}\right)
= \frac{2\,\cA(\f_\tau)}{|\tau|} \, \nabla\cA(\f_\tau).
$$
Equating this with $2\,L_S(\f_\tau)\,\f_\tau$ and dividing by $2$, we obtain
$$
L_S(\f_\tau)\,\f_\tau 
= \frac{\cA(\f_\tau)}{|\tau|}\,\nabla\cA(\f_\tau),
$$
which is exactly \eqref{eq:GradA}.
\end{proof}
More details about the calculation of $\nabla\cA$ and its derivative are reported in \protect{\cite[Appendix~A]{Sutti_Yueh:2024}}.

\section{Fundamental domain and cohomology form} \label{sec:fundamental_domain}

A fundamental domain of a genus-one closed surface $\cM$ is a bounded, simply connected planar region $\cD\subset\R^2$ whose two pairs of opposite boundary curves are identified by linearly independent translation vectors $\mathbf{w}_1,\mathbf{w}_2\in\R^2$. The translations generate the lattice
$$
\Lambda = \{ k_{1}\mathbf{w}_1+k_2\mathbf{w}_2 \mid k_1, k_2\in\mathbb{Z}\}
$$
which tiles the plane, and the resulting quotient surface $\cD/\Lambda$ is compact and homeomorphic to $\cM$. This fundamental domain thus supplies the global coordinate chart on which we construct the initial area-preserving torus map. Later, in Figure~\ref{fig:vertebra_and_torus}(c), we show the fundamental domain for the simplicial surface named Vertebrae, which is one of the benchmark mesh models considered in this paper.

To compute the fundamental domain $\cD$ for the genus-one surface $\cM$, we first apply the \texttt{ReebHanTun} algorithm~\cite{Dey:2008} to extract two simple, independent, non-contractible, directed loops $\gamma_1$, $\gamma_2$ that intersect at one vertex. 
For each loop $\gamma_\ell$, we build an integer-valued closed 1-form 
$$
\eta_\ell([\bv_i,\bv_j]) = \begin{cases}
1 & \text{ if $\bv_i\in \gamma_\ell$ and $\bv_j$ on the left-hand side of $\gamma_\ell$,} \\
-1 & \text{ if $\bv_j\in \gamma_\ell$ and $\bv_i$ on the left-hand side of $\gamma_\ell$,} \\
0 & \text{ otherwise.}
\end{cases}
$$
Solving a cotangent-weighted Poisson equation, i.e.,
$$
\sum_{\bv_j\in N(\bv_i)} \frac{\cot\theta_{i,j} + \cot\theta_{j,i}}{2} \left( \eta_\ell([\bv_i,\bv_j]) + h_\ell(\bv_j) - h_\ell(\bv_i) \right) = 0,
$$
with $\theta_{i,j}$ and $\theta_{j,i}$ being the angles opposite edge $[\bv_i,\bv_j]$, produces a harmonic 1-form 
$$
\omega_\ell([\bv_i,\bv_j]) = \eta_\ell([\bv_i,\bv_j]) + h_\ell(\bv_j) - h_\ell(\bv_i), \quad \ell=1,2.
$$
Each harmonic 1-form $\omega_{\ell}$ defines a holomorphic 1-form $\zeta_\ell=\omega_\ell+i\star\omega_\ell$, where $\star$ denotes the Hodge operator.  
After cutting the mesh along $\gamma_{1}\cup\gamma_{2}$, we integrate an appropriate linear combination $\zeta = c_1 \zeta_1 + c_2 \zeta_2$ from a root vertex $\bv_1$ to every other vertex $\bv_k$ as $g(\bv_k) = \int_{\bv_1}^{\bv_k} \zeta$. The resulting image $g(\cM)$ of the complex-valued mapping is the desirable fundamental domain $\cD$. Algorithm~\ref{algo:fundamental_domain} gives a pseudocode for computing the fundamental domain; more computational details can be found in~\protect{\cite{Yueh:2020}}.

\begin{algorithm}
\SetAlgoLined
 Given a genus-one closed surface $\cM$\;
 \KwResult{Fundamental domain $\cD$.}
  Apply the \texttt{ReebHanTun} algorithm to extract $\gamma_1$, $\gamma_2$\;
  For each $ \gamma_{\ell} $, $ \ell = 1,2 $, build an integer-valued closed 1-form $ \eta_{\ell} $\;
  Compute the harmonic 1-form $\omega_{\ell}$ by solving a cotangent-weighted Poisson equation\;
  Compute the holomorphic 1-forms $\zeta_\ell=\omega_\ell+i\star\omega_\ell$\;
  Slice the mesh along $\gamma_{1}\cup\gamma_{2}$\;
  Integrate an appropriate linear combination $\zeta = c_1 \zeta_1 + c_2 \zeta_2$ from a root vertex $\bv_1$ to every other vertex $\bv_k$ as $g(\bv_k) = \int_{\bv_1}^{\bv_k} \zeta$\; 
 \textbf{return} $\cD = g(\cM) $.
 \caption{Calculation of the fundamental domain.}\label{algo:fundamental_domain}
\end{algorithm}

\section{Riemannian optimization framework and geometry} \label{sec:riem_optim_and_geometry}

The \textit{Riemannian optimization framework} \protect{\cite{EAS:1998,AMS:2008,Boumal:2023}} solves constrained optimization problems where the constraints have a geometric structure, allowing the constraints to be considered explicitly. More precisely, the optimization variables are constrained to a smooth manifold, and the optimization is performed on that manifold. Typically, the manifolds considered are matrix manifolds, meaning there is a natural representation of their elements in matrix form. In particular, in this paper, the optimization variable is constrained to a power manifold of $n$ ring tori embedded in $\R^{3}$.

Generally speaking, a line-search method in the Riemannian framework determines at a current iterate $\x_{k} $ on a manifold $M$ a search direction $\bd_{k}$ on the tangent space $\mathrm{T}_{\x_{k}} M$. The next iterate $\x_{k+1}$ is then determined by a line search along a curve $\alpha \mapsto \Retraction_{\x_{k}}(\alpha \bd_{k})$ where $\Retraction_{\x_{k}} \colon \mathrm{T}_{\x_{k}} M \to M$ is the retraction mapping.
The procedure is then repeated for $\x_{k+1}$ taking the role of $\x_{k}$. Similarly to optimization methods in Euclidean space, search directions can be the negative of the Riemannian gradient, leading to the Riemannian steepest descent method. Other choices of search directions lead to different methods, e.g., Riemannian versions of the trust-region method~\protect{\cite{ABG:2007}} or the (limited-memory) BFGS method~\protect{\cite{RingWirth:2012}}.

In what follows, we introduce some fundamental geometry concepts necessary to formulate the algorithms. We first describe the geometry of the ring torus $\Torus$ embedded in $\R^{3}$, and then we switch to the power manifold of $n$ ring tori, denoted by $\powerTorus$.

\subsection{Geometry of the ring torus \texorpdfstring{$\Torus$}{TEXT}}

This section describes the geometry of the ring torus, including tools such as the projection of a point onto the torus, the projection onto the tangent space, the retraction and the parallel transport of tangent vectors.

Let $ \Sone $, $ \mathbb{S}^{1}_{R} $ be two circles of minor radius $ r $ and major radius $R>r$, respectively.
The (ring) torus can be regarded as a Cartesian product of the two circles: $ \Torus = \Sone \times \mathbb{S}^{1}_{R} $, i.e., it is a surface of revolution generated by rotating the circle of minor radius $\Sone$ around the circle of major radius $\mathbb{S}^{1}_{R}$.

A generic point $\mathbf{p}$ of a torus $ \Torus $ has coordinates
\[
   \begin{cases}
      p_{1} = ( R + r\cos\phi ) \cos\theta, \\
      p_{2} = ( R + r\cos\phi ) \sin\theta, \\
      p_{3} = r\sin\phi,
   \end{cases}
\]
where the azimuthal angle is $\theta \in [0,2\pi)$ and the altitude (or elevation) angle is $\phi \in [0,2\pi)$.

Since we wish every vertex of the mesh to be constrained to a torus, our optimization problem will be formulated on a Cartesian product of $n$ ring tori $\Torus$, i.e.,
\[
   \powerTorus = \underbrace{\Torus \times \dots \times \Torus}_{n~\text{times}},
\]
which we also call power manifold of $n$ ring tori.
This is in analogy to what we did in our previous work~\protect{\cite{Sutti_Yueh:2024}}, where we considered optimization on a power manifold of unit spheres. Before discussing the power manifold $ \powerTorus $, we dive deeper into the geometric tools of $ \Torus $.

\subsubsection{Projection of a point onto the torus \texorpdfstring{$\Torus$}{TEXT}} \label{sec:proj_point_onto_torus}

Let $ \q = ( q_{1}, q_{2}, q_{3} ) $ be a generic point of $\R^{3}$. Let $\q'$ denote the point at the intersection between $ \mathbb{S}^{1}_{R} $ and the vertical plane passing through $\q$ and the origin; see Figure~\ref{fig:torus_cross_section}. Then the coordinates of $\q'$ are  $ ( R\cos\theta_{\q}, \ R\sin\theta_{\q}, \ 0 ) $, where $\theta_{\q}$ is the angle between $\q$ and the $xy$-plane. 

The projection of a generic point $\q \in \R^{3}$ onto $ \Torus $ is
\begin{equation}\label{eq:proj_onto_torus}
   \widetilde{\q} = \Pi_{\Torus}(\q) =
   \begin{pmatrix}
      ( R + r\cos\phi_{\q} ) \cos\theta_{\q} \\
      ( R + r\cos\phi_{\q} ) \sin\theta_{\q} \\
      r\sin\phi_{\q}
   \end{pmatrix}.
\end{equation}
We calculate the values of $\cos\theta_{\q}$ and $\sin\theta_{\q}$ directly without passing from the angle $\theta_{\q}$, i.e.,
\begin{equation}\label{eq:cos_sin_theta}
   \cos\theta_{\q} = \frac{q_{1}}{\sqrt{q_{1}^{2}+q_{2}^{2}}}, \qquad \sin\theta_{\q} = \frac{q_{2}}{\sqrt{q_{1}^{2}+q_{2}^{2}}}.
\end{equation}
Similarly, we write $\cos\phi_{\q}$ and $\sin\phi_{\q}$ directly without passing from $\phi_{\q}$, namely, 
\begin{equation}\label{eq:cos_sin_phi}
   \cos\phi_{\q} = \frac{c}{\sqrt{c^{2}+q_{3}^{2}}}, \qquad \sin\phi_{\q} = \frac{q_{3}}{\sqrt{c^{2}+q_{3}^{2}}},
\end{equation}
where $ c \coloneqq \sqrt{q_{1}^{2}+q_{2}^{2}}-R $. These calculations are formalized by Algorithm~\ref{algo:projection_on_torus}, and the auxiliary Algorithms~\ref{algo:cos_sin_theta} and \ref{algo:cos_sin_phi}.

\begin{algorithm}
\SetAlgoLined
 Given point $ \q \in \R^{3} $, torus $ \Torus $\;
 \KwResult{Projection $ \widetilde{\q} \equiv (\widetilde{q}_{1},\widetilde{q}_{2},\widetilde{q}_{3}) $ of $ \q $ onto $ \Torus $.}
 Call Algorithm~\ref{algo:cos_sin_theta} to compute $ \cos \theta_{\q} $ and $ \sin \theta_{\q} $\;
 Call Algorithm~\ref{algo:cos_sin_phi} to compute $ \cos \phi_{\q} $ and $ \sin \phi_{\q} $\;
 $ \widetilde{R} \leftarrow R + r \cos\phi_{\q} $\;
 $ \widetilde{q}_{1} \leftarrow \widetilde{R} \cos\theta_{\q} $\;
 $ \widetilde{q}_{2} \leftarrow \widetilde{R} \sin\theta_{\q} $\;
 $ \widetilde{q}_{3} \leftarrow r \sin\phi_{\q} $\;
 \textbf{return} $ \widetilde{\q} \equiv (\widetilde{q}_{1},\widetilde{q}_{2},\widetilde{q}_{3}) $.
 \caption{Projection of a point $ \q $ onto the torus $ \Torus $.}\label{algo:projection_on_torus}
\end{algorithm}

\begin{algorithm}
\SetAlgoLined
 Given point $ \q \in \R^{3} $, torus $ \Torus $\;
 \KwResult{$ \cos \theta_{\q} $ and $ \sin \theta_{\q} $.}
 $ \mathrm{den}_{\theta} \leftarrow \sqrt{q_{1}^{2}+q_{2}^{2}} $\;
 $ \cos \theta_{\q} \leftarrow q_{1}/\mathrm{den}_{\theta} $\;
 $ \sin \theta_{\q} \leftarrow q_{2}/\mathrm{den}_{\theta} $\;
 \textbf{return} $ \cos \theta_{\q} $ and $ \sin \theta_{\q} $.
 \caption{Compute $ \cos \theta_{\q} $ and $ \sin \theta_{\q} $ of a point $ \q \in \Torus $.}\label{algo:cos_sin_theta}
\end{algorithm}

\begin{algorithm}
\SetAlgoLined
 Given point $ \q \in \R^{3} $, torus $ \Torus $\;
 \KwResult{$ \cos \phi_{\q} $ and $ \sin \phi_{\q} $.}
 $ \mathrm{den}_{\theta} \leftarrow \sqrt{q_{1}^{2}+q_{2}^{2}} $\;
 $ c \leftarrow \mathrm{den}_{\theta} - R $\;
 $ \mathrm{den}_{\phi} \leftarrow \sqrt{c^{2}+q_{3}^{2}} $\;
 $ \cos \phi_{\q} \leftarrow c/\mathrm{den}_{\phi} $\;
 $ \sin \phi_{\q} \leftarrow q_{3}/\mathrm{den}_{\phi} $\;
 \textbf{return} $ \cos \phi_{\q} $ and $ \sin \phi_{\q} $.
 \caption{Compute $ \cos \phi_{\q} $ and $ \sin \phi_{\q} $ of a point $ \q \in \Torus
 $.}\label{algo:cos_sin_phi}
\end{algorithm}

\begin{figure}[htbp]
  \centering
  \includegraphics[width=0.8\textwidth]{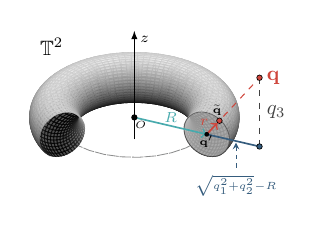}
  \caption{A torus cross-section illustrating the projection of a point $ \q $ onto the torus $\Torus$. This image has been adapted from \url{https://tikz.net/torus/}.}\label{fig:torus_cross_section}
\end{figure}

\subsubsection{Projection of a point onto the tangent space to \texorpdfstring{$\Torus$}{TEXT} at \texorpdfstring{$\f_{\ell}$}{TEXT}}

Let $ \q = ( q_{1}, q_{2}, q_{3} ) $ be a point of $\R^{3}$, and let $ \f_{\ell} = \left( f_\ell^1, f_\ell^2, f_\ell^3 \right)\tr $ be a point of $\Torus$, consistently with the notation introduced in Sect.~\ref{sec:simplicial_surfaces}.
The projection of $ \q $ onto the tangent space at $\f_{\ell}$ to $\Torus$ is computed as follows.

\begin{enumerate}
    \item Call Algorithm~\ref{algo:cos_sin_theta} to compute cosine and sine of the azimuthal angle $ \theta_{\f_{\ell}} $, i.e.,
    \[
       \cos\theta_{\f_{\ell}} = \frac{f_{\ell}^{1}}{\sqrt{(f_{\ell}^{1})^{2}+(f_{\ell}^{2})^{2}}}, \qquad \sin\theta_{\f_{\ell}} = \frac{f_{\ell}^{2}}{\sqrt{(f_{\ell}^{1})^{2}+(f_{\ell}^{2})^{2}}}.
    \]
    \item Compute the coordinates of the center of the circle $ \Sone $ given by the intersection between the torus and the vertical plane that passes through the points $ \q $ and the origin, i.e.,
    \[
       \bc = \left( R\cos\theta_{\f_{\ell}}, \, R\sin\theta_{\f_{\ell}}, \, 0 \right).
    \]
    \item Translate the points $\f_{\ell}$ and $\q$ to the circle $ \Sone $ centered at the origin, i.e.,
    \[
       \widehat{\f_{\ell}} = \f_{\ell} - \bc, \qquad \widehat{\q} = \q - \bc.
    \]
    \item Project $\widehat{\q}$ onto the tangent space to $\Sone$ at $\f_{\ell}$,
    \[
       \P_{\mathrm{T}_{\widehat{\f_{\ell}}}\Sone} \! \left( \widehat{\q} \right) = \left( I - \frac{\widehat{\f_{\ell}}\widehat{\f_{\ell}}\tr}{\widehat{\f_{\ell}}\tr\widehat{\f_{\ell}}} \right) \widehat{\q} = \widehat{\q} - \frac{\widehat{\f_{\ell}}\tr\widehat{\q}}{\widehat{\f_{\ell}}\tr\widehat{\f_{\ell}}} \, \widehat{\f_{\ell}} + \widehat{\f_{\ell}}.
    \]
    \item Translate the result back to the ``original position'' on the torus, obtaining the sought projection:
    \[
       \P_{\mathrm{T}_{\f_{\ell}}\Torus} \! \left( \q \right) \coloneqq \P_{\mathrm{T}_{\widehat{\f_{\ell}}}\Sone} \! \left( \widehat{\q} \right) + \bc.
    \]
\end{enumerate}
Algorithm~\ref{algo:proj_onto_tg_space} gives a pseudocode for the projection of a point $\q$ onto the tangent space $ \mathrm{T}_{\f_{\ell}}\Torus $.

\begin{algorithm}
\SetAlgoLined
 Given point $ \q \in \R^{3} $, torus $ \Torus $, point $ \f_{\ell} \in \Torus $, and tangent space $ \mathrm{T}_{\f_{\ell}}\Torus $\;
 \KwResult{Projection $ \P_{\mathrm{T}_{\f_{\ell}}\Torus} \! \left( \q \right) $ of $ \q $ onto $ \mathrm{T}_{\f_{\ell}}\Torus $.}
 Call Algorithm~\ref{algo:cos_sin_theta} to compute $ \cos \theta_{\f_{\ell}} $ and $ \sin \theta_{\f_{\ell}} $\;
 $ c_{1} \leftarrow R \cos\theta_{\f_{\ell}} $\;
 $ c_{2} \leftarrow R \sin\theta_{\f_{\ell}} $\;
 $ c_{3} \leftarrow 0 $\;
 $ \widehat{\f_{\ell}} \leftarrow \f_{\ell} - \bc $\;
 $ \widehat{\q} \leftarrow \q - \bc $\;
 $ \P_{\mathrm{T}_{\widehat{\f_{\ell}}}\Sone} \! \left( \widehat{\q} \right) \leftarrow \widehat{\q} - \frac{\widehat{\f_{\ell}}\tr\widehat{\q}}{\widehat{\f_{\ell}}\tr\widehat{\f_{\ell}}} \, \widehat{\f_{\ell}} + \widehat{\f_{\ell}} $\;
 \textbf{return} $ \P_{\mathrm{T}_{\f_{\ell}}\Torus} \! \left( \q \right) \leftarrow \P_{\mathrm{T}_{\widehat{\f_{\ell}}}\Sone} \! \left( \widehat{\q} \right) + \bc $.
 \caption{Projection of a point $\q$ onto the tangent space $ \mathrm{T}_{\f_{\ell}}\Torus $.}\label{algo:proj_onto_tg_space}
\end{algorithm}

\subsubsection{Retraction onto \texorpdfstring{$\Torus$}{TEXT}}

A retraction is a mapping from the tangent space to the manifold used to turn tangent vectors into points of the manifold, and functions defined on the manifold into functions defined on the tangent space; see \protect{\cite{AMS:2008,Absil:2012}} for more details. The key idea relevant to our work is that, for any embedded submanifold, a simple retraction is given by taking a tangent vector step from a given point of the manifold into the embedding space, followed by a projection onto the manifold; see, e.g., \protect{\cite[Prop.~3.6.1]{AMS:2008}}.

The retraction of a vector $ \boldxi \in \mathrm{T}_{\x}\Torus $ from the tangent space $ \mathrm{T}_{\x}\Torus $ to the torus $ \Torus $ is calculated by moving from $ \x $ in the direction of $ \boldxi $ in the embedding space $\R^{3}$, and then projecting $\x + \boldxi$ onto the torus $\Torus$ using \eqref{eq:proj_onto_torus}, i.e.,
\begin{equation}\label{eq:retraction_on_torus}
   \Retraction_{\x}(\boldxi) \coloneqq \Pi_{\Torus}(\x + \boldxi) =
   \begin{pmatrix}
      ( R + r\cos\phi_{\x + \boldxi} ) \cos\theta_{\x + \boldxi} \\
      ( R + r\cos\phi_{\x + \boldxi} ) \sin\theta_{\x + \boldxi} \\
      r\sin\phi_{\x + \boldxi}
   \end{pmatrix},
\end{equation}
where the angles are computed using the formulas \eqref{eq:cos_sin_theta} and \eqref{eq:cos_sin_phi}, with $\x + \boldxi$ taking the role of $\q$. In other words, the algorithm for the retraction at $ \x $ of $ \boldxi $ onto $\Torus$ is given by Algorithm~\ref{algo:projection_on_torus} applied to $\x + \boldxi$.

\subsubsection{Parallel transport}\label{sec:parallel_transp}
Parallel transport enables the consistent movement of vectors between tangent spaces. In this paper, parallel transport on the torus is employed within the Riemannian conjugate gradient method discussed in Sect.~\ref{sec:rcg_method}. From a theoretical perspective, parallel transport is crucial for formulating the Lipschitz condition on Riemannian gradients and establishing convergence guarantees.

Formally, parallel transport is defined as follows. Given a Riemannian manifold $(M,g)$ and two points $\x,\y\in M$, the parallel transport $\mathcal{T}_{\x\to\y}\colon \mathrm{T}_{\x}M \to \mathrm{T}_{\y}M$ is a linear operator that preserves the inner product between two tangent vectors, namely
\[
   \forall\boldxi, \boldsymbol{\zeta} \in \mathrm{T}_{\x}M, \qquad \langle \mathcal{T}_{\x\to \y}\xi, \mathcal{T}_{\x\to \y}\boldsymbol{\zeta} \rangle_{\y}=\langle \boldxi, \boldsymbol{\zeta} \rangle_{\x}.
\]
In general, computing parallel transports involves numerically solving ordinary differential equations (ODEs). This process requires explicitly selecting a (possibly geodesic) curve that connects points $\x$ and $\y$. To determine a minimizing geodesic, one must compute the Riemannian logarithm. As a result, the computation of parallel transports can be quite costly in practice.

To address these computational challenges for our specific case, we explicitly construct the parallel transport on the torus, exploiting the fact that the torus $ \Torus $ is given by the Cartesian product of two circles, i.e., $ \Torus = \Sone \times \mathbb{S}^{1}_{R} $.

\subsubsection{Parallel transport on \texorpdfstring{$\Torus$}{TEXT}}

Given a torus $ \Torus $, two points $ \x, \, \y \in \Torus $, and a tangent vector $ \boldxi_{\x} \in \mathrm{T}_{\x}\Torus $, our aim is to transport $\boldxi_{\x}$ to the tangent space $\mathrm{T}_{\y}\Torus$. Exploiting the Cartesian product structure of $\Torus$, this can be achieved via two rotations and two translations, as follows.

\begin{enumerate}
    \item Rotate $\boldxi_{\x}$ by an angle $\theta_{\y} - \theta_{\x}$ with respect to the $z$ axis, obtaining $\boldxi'$:
    \[
       \boldxi' = \mathbf{R}_{\theta_{\y} - \theta_{\x}, z} \, \boldxi_{\x} =  \begin{pmatrix}
       \cos\theta_{\mathrm{diff}} \, \xi_{1} - \sin\theta_{\mathrm{diff}} \, \xi_{2} \\
       \sin\theta_{\mathrm{diff}} \, \xi_{1} + \cos\theta_{\mathrm{diff}} \, \xi_{2} \\
       \xi_{3}
       \end{pmatrix},
    \]
    where
    \[
        \cos\theta_{\mathrm{diff}} \coloneqq \cos( \theta_{\y} - \theta_{\x} ) = \cos\theta_{\y}\cos\theta_{\x} + \sin\theta_{\y}\sin\theta_{\x},
    \]
    \[
       \sin\theta_{\mathrm{diff}} \coloneqq \sin( \theta_{\y} - \theta_{\x} ) = \sin\theta_{\y}\cos\theta_{\x} - \cos\theta_{\y}\sin\theta_{\x},
    \]
    with
    \[
       \cos\theta_{\x} = \frac{x_{1}}{\sqrt{x_{1}^{2}+x_{2}^{2}}}, \qquad \sin\theta_{\x} = \frac{x_{2}}{\sqrt{x_{1}^{2}+x_{2}^{2}}}.
    \]
    \item Translate $\boldxi'$ to the origin: $\boldxi''=\boldxi'-\y'$, where $\y'$ is given by $ \y' = (R\cos\theta_{\y},R\sin\theta_{\y},0) $, with
    \[
       \cos\theta_{\y} = \frac{y_{1}}{\sqrt{y_{1}^{2}+y_{2}^{2}}}, \qquad \sin\theta_{\y} = \frac{y_{2}}{\sqrt{y_{1}^{2}+y_{2}^{2}}}.
    \]
    \item Rotate $\boldxi''$ with Rodrigues' rotation formula by an angle $\phi_{\y} - \phi_{\x}$ with respect to the axis $\mathbf{k} \coloneqq (-\sin\theta_{\y'},\cos\theta_{\y'},0)$ (this is the tangent vector at $\y'$ to the main tunnel of the torus of radius $R$), obtaining $\boldxi'''$
    \[
       \boldxi''' = \mathbf{R}_{\phi_{\y} - \phi_{\x}, \mathbf{k}}\,\boldxi'' =  \boldxi'' + \sin\phi_{\mathrm{diff}} \, (\mathbf{k} \times \boldxi'') - (1-\cos\phi_{\mathrm{diff}} ) \, (\mathbf{k}\times \boldxi'') \times \mathbf{k},
    \]
where $ \times $ denotes the standard cross product on two vectors, and
    \[
       \sin\phi_{\mathrm{diff}} \coloneqq \sin( \phi_{\y} - \phi_{\x} ) = \sin\phi_{\y}\cos\phi_{\x} - \cos\phi_{\y}\sin\phi_{\x},
    \]
    \[
       \cos\phi_{\mathrm{diff}} \coloneqq \cos( \phi_{\y} - \phi_{\x} ) = \cos\phi_{\y}\cos\phi_{\x} + \sin\phi_{\y}\sin\phi_{\x},
    \]
    \[
       \cos\phi_{\x} = \frac{c_{\x}}{\sqrt{c_{\x}^{2}+x_{3}^{2}}}, \quad \sin\phi_{\x} = \frac{x_{3}}{\sqrt{c_{\x}^{2}+x_{3}^{2}}},
       \quad \text{with} \quad c_{\x} = \sqrt{x_{1}^{2}+x_{2}^{2}}-R,
    \]
    \[
       \cos\phi_{\y} = \frac{c_{\y}}{\sqrt{c_{\y}^{2}+y_{3}^{2}}}, \quad \sin\phi_{\y} = \frac{y_{3}}{\sqrt{c_{\y}^{2}+y_{3}^{2}}},
       \quad \text{with} \quad c_{\y} = \sqrt{y_{1}^{2}+y_{2}^{2}}-R,
    \]
    \item Translate $\boldxi'''$ back to the torus: $\boldxi_{\y} = \boldxi'''+\y'$.
\end{enumerate}

This procedure is formalized in Algorithm~\ref{algo:parallel_transport}.

\begin{algorithm}
\SetAlgoLined
 Given torus $ \Torus $, points $ \x, \, \y \in \Torus $, and tangent vector $ \boldxi_{\x} \in \mathrm{T}_{\x}\Torus $\;
 \KwResult{Parallel-transported vector $ \boldxi_{\y} \in \mathrm{T}_{\y}\Torus $.}
 Call Algorithm~\ref{algo:cos_sin_theta} to compute $ \cos \theta_{\x} $ and $ \sin \theta_{\x} $\;
 Call Algorithm~\ref{algo:cos_sin_theta} to compute $ \cos \theta_{\y} $ and $ \sin \theta_{\y} $\;
 $ \cos\theta_{\mathrm{diff}} = \cos\theta_{\y}\cos\theta_{\x} + \sin\theta_{\y}\sin\theta_{\x} $\;
 $ \sin\theta_{\mathrm{diff}} = \sin\theta_{\y}\cos\theta_{\x} - \cos\theta_{\y}\sin\theta_{\x} $\;
    \[
       \boldxi' = \mathbf{R}_{\theta_{\y} - \theta_{\x}, z} \, \boldxi = \begin{pmatrix}
      \cos\theta_{\mathrm{diff}} \, \xi_{1} - \sin\theta_{\mathrm{diff}} \, \xi_{2} \\
      \sin\theta_{\mathrm{diff}} \, \xi_{1} + \cos\theta_{\mathrm{diff}} \, \xi_{2} \\
      \xi_{3}
   \end{pmatrix};
    \]\\
 $ \y' = (R\cos\theta_{\y},R\sin\theta_{\y},0) $\;
 Call Algorithm~\ref{algo:cos_sin_phi} to compute $ \cos \phi_{\x} $ and $ \sin \phi_{\x} $\;
 Call Algorithm~\ref{algo:cos_sin_phi} to compute $ \cos \phi_{\y} $ and $ \sin \phi_{\y} $\;
 $ \mathbf{k} = (-\sin\theta_{\y'},\cos\theta_{\y'},0) $\;
 $ \cos\phi_{\mathrm{diff}} = \cos\phi_{\y}\cos\phi_{\x} + \sin\phi_{\y}\sin\phi_{\x} $\;
 $ \sin\phi_{\mathrm{diff}} = \sin\phi_{\y}\cos\phi_{\x} - \cos\phi_{\y}\sin\phi_{\x} $\;
 $\boldxi''=\boldxi'-\y'$\;
  Use Rodrigues’ rotation formula: $ \boldxi''' = \mathbf{R}_{\phi_{\y} - \phi_{\x}, \mathbf{k}}\,\boldxi'' =  \boldxi'' +  \sin\phi_{\mathrm{diff}} \, (\mathbf{k} \times \boldxi'') - (1- \cos\phi_{\mathrm{diff}}) (\mathbf{k}\times \boldxi'') \times \mathbf{k} $\; 
 Translate $\boldxi'''$ back to the torus: $\boldxi_{\y} = \boldxi'''+\y'$.
 \caption{Parallel transport on the torus $ \Torus $.}\label{algo:parallel_transport}
\end{algorithm}

\subsection{Geometry of the power manifold \texorpdfstring{$\powerTorus$}{TEXT}{}}\label{sec:power_manifold}

As mentioned earlier, we aim to minimize the function $E(f) = E(\f_{1},\dots,\f_{n})$, defined in~\eqref{eq:objective_function}, where each point $\f_{\ell}$, $\ell=1,\dots,n$, lives on the same manifold $\Torus$. This leads us to consider the \emph{power manifold} of $n$ tori
\[
   \powerTorus = \underbrace{\Torus \times \dots \times \Torus}_{n~\text{times}},
\]
with the metric of $\Torus$ extended elementwise. The tools that we use to create optimization algorithms on this power manifold are straightforward elementwise extensions of the same tools on the torus $ \Torus $. Given a power torus $\powerTorus$ and a point $\f \in \powerTorus$, the projection onto the tangent space $ \P_{\mathrm{T}_{\f}\powerTorus} \colon \R^{n \times 3} \to \mathrm{T}_{\f}\powerTorus $ is used to compute the Riemannian gradient, as explained in Sect.~\ref{sec:Riem_grad_descent}. The projection onto the power torus $\Pi_{\powerTorus} \colon \R^{n \times 3} \to \powerTorus $ turns points of $\R^{n\times 3}$ into points of $\powerTorus$. Finally, the retraction $\Retraction_{\f} \colon \mathrm{T}_{\f}\powerTorus \to \powerTorus $ maps tangent vectors of defined on $ \mathrm{T}_{\f}\powerTorus $ into points on $\powerTorus$.

\section{Riemannian optimization algorithms}\label{sec:riemannian_algorithms}

We describe the algorithms used in this paper: Riemannian gradient descent, projected gradient descent, Riemannian conjugate gradient, and projected conjugate gradient. We regard the projected gradient method as an approximation of the Riemannian gradient method, and the projected conjugate gradient method as an approximation of the Riemannian conjugate gradient method. These algorithms share the same algorithmic components, mainly projections and retractions.

In all the cases, the initial torus mapping is computed by minimizing the authalic energy of the mapping in variables of the planar fundamental domain \protect{\cite{Yueh:2020}}, which is a modification of the holomorphic differential method introduced in \protect{\cite{GuYa02}}.

\subsection{Riemannian gradient descent method on \texorpdfstring{$\powerTorus$}{TEXT}} \label{sec:Riem_grad_descent}

The Riemannian gradient descent (RGD) method is the Riemannian generalization of the steepest (or gradient) descent method. The main feature of the Riemannian gradient descent method is the calculation of the Riemannian gradient of the objective function as a projection of the Euclidean gradient onto the tangent space $ \mathrm{T}_{\f^{(k)}}\Torus $; see Figure~\ref{fig:torus_riem_grad}. 

In the RGD method, the descent direction is defined as the negative of the Riemannian gradient; the new iterate is computed by a line searching along this direction and using retractions.
Algorithm~\ref{algo:RGD} provides a pseudocode for the RGD method on the power torus $\powerTorus$.

\begin{figure}[htbp]
  \centering
  \includegraphics[width=0.55\textwidth]{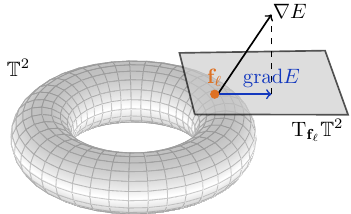}
\caption{Euclidean and Riemannian gradients of a function $ E \colon \Torus \to \R $.}
\label{fig:torus_riem_grad}
\end{figure}

\begin{algorithm}
\SetAlgoLined
 Given objective function $E$, power manifold $\powerTorus$, initial iterate $\f^{(0)}\in \powerTorus$, projector $ \P_{\mathrm{T}_{\f}\powerTorus} $ from $\R^{n\times 3}$ to $\mathrm{T}_{\f}\powerTorus$, retraction $\Retraction_{\f}$ from $\mathrm{T}_{\f}\powerTorus$ to $\powerTorus$\;
 \KwResult{Sequence of iterates $\lbrace f^{(k)} \rbrace$.}
 $k \leftarrow 0$\;
 \While{$f^{(k)}$ does not sufficiently minimizes $E$}{
     Compute the Euclidean gradient of the objective function $\nabla E(f^{(k)})$ \;
     Compute the Riemannian gradient as $ \grad E(f^{(k)}) = \P_{\f^{(k)}} \! \big( \nabla E(f^{(k)}) \big)$\;
     Choose the anti-gradient direction $ \boldd^{(k)} = -\grad E(f^{(k)}) $\;
     Compute a step size $\alpha_{k} > 0$ with line-search that satisfies the sufficient decrease condition\;
     Set $\f^{(k+1)} = \Retraction_{\f^{(k)}}(\alpha_{k} \boldd^{(k)})$\;
     $ k \leftarrow k + 1 $\;
 }
 \caption{The RGD method on $\powerTorus$.}\label{algo:RGD}
\end{algorithm}

The RGD method has theoretically guaranteed convergence: we refer the reader to~\protect{\cite[\S 4]{Sutti_Yueh:2024}} and references therein for a review of the known convergence results.
See also \protect{\cite[Theorem 4.20]{Boumal:2023}}.

\subsection{Projected gradient method}
The projected gradient method (PGM) is like the classical gradient method in Euclidean space, with the additional step of the projection onto the manifold, as discussed in Sect.~\ref{sec:proj_point_onto_torus} and summarized in Algorithm~\ref{algo:projection_on_torus}.
Given an iterate $\f^{(k)}$ in $\powerTorus$, the step $\alpha_{k}$ is obtained by searching the path
\[
   \f^{(k)}(\alpha) \coloneqq \Pi_{\powerTorus}\left( \f^{(k)} - \alpha \nabla E(\f^{(k)}) \right),
\]
where $ \Pi_{\powerTorus} \colon \R^{n \times 3} \to \powerTorus $ is the projection onto $ \powerTorus $, as discussed in Sect.~\ref{sec:proj_point_onto_torus}.
Given a step $ \alpha_{k} $, computed according to the line-search technique discussed in Appendix~\ref{sec:line_search}, the next iterate is defined by
\[
   \f^{(k+1)} \coloneqq \f^{(k)}(\alpha_k) = \Pi_{\powerTorus}\left( \f^{(k)} - \alpha_{k} \nabla E(\f^{(k)}) \right).
\]

The pseudocode for the projected gradient method is given in Algorithm~\ref{algo:PG}.

\begin{algorithm}
\SetAlgoLined
 Given objective function $E$, power manifold $\powerTorus$, initial iterate $\f^{(0)}\in \powerTorus$, projector $ \Pi_{\powerTorus} $ from $\R^{n\times 3}$ to $ \powerTorus $\;
 $k \leftarrow 0$\;
 \While{$f^{(k)}$ does not sufficiently minimizes $E$}{
     Compute the Euclidean gradient of the objective function $\nabla E(f^{(k)})$\;
     Set  $ \boldd^{(k)} = -\nabla E(f^{(k)}) $\;
     Compute a step size $\alpha_{k} > 0$ with a line-search procedure that satisfies the sufficient decrease condition\;
     Set $ \f^{(k+1)} = \Pi_{\powerTorus} \big(\f^{(k)} + \alpha_{k} \boldd^{(k)} \big) $ using Algorithm~\ref{algo:projection_on_torus}\;
     $ k \leftarrow k + 1 $\;
 }
 \caption{The projected gradient method on $\powerTorus$.}\label{algo:PG}
\end{algorithm}

For theoretical purposes, we may regard the PGM as an approximate version of the RGD method. The main differences between the two algorithms are that:
\begin{itemize}
    \item PGM uses the Euclidean gradient, not the Riemannian gradient, so the line-search procedure is performed in the embedding space $ \R^{n \times 3} $.
    \item The new iterate in PGM is computed projecting directly onto $\powerTorus$; there is no intermediate step involving the tangent space.
\end{itemize}
The numerical experiments of Sect.~\ref{sec:numerical_experiments} show that, in practice, the PGM always converges.

\subsection{Riemannian conjugate gradient method}\label{sec:rcg_method}

The Riemannian conjugate gradient (RCG) method is discussed in~\protect{\cite{RingWirth:2012}}; we also refer the reader to the more recent works of Sato et al.~\protect{\cite{SatoIwai:2015,Sato:2016,Sato:2021,Sato:2022}} for a comprehensive study of RCG methods.
Besides the components from Riemannian geometry that are already used in RGD, the main addition is the introduction of parallel transport.

The RCG method uses the steepest descent direction at the first iteration. Then, for all the subsequent iterations, the update direction is chosen to be a linear combination of the Riemannian gradient at the current iterate and the parallel-transported previous direction. It is formalized as follows
\[
   \boldd^{(k)} = -\grad E(f^{(k)}) + \beta_{k}^{\mathrm{FR}} \, \mathcal{T}_{\x\to \y} \, \boldd^{(k-1)},
\]
where the scalar $\beta_{k}^{\mathrm{FR}}$ is computed as the ratio of two norms
\begin{equation}\label{eq:fletcher_reeves}
   \beta_{k}^{\mathrm{FR}} = \frac{\|\grad E(f^{(k)})\|^{2}}{\|\grad E(f^{(k-1)})\|^{2}},
\end{equation}
and $ \mathcal{T}_{\x\to \y}\colon \mathrm{T}_{\x}\Torus \to \mathrm{T}_{\y}\Torus $ is the parallel transport from $\mathrm{T}_{\x}\Torus$ to $\mathrm{T}_{\x}\Torus$, described in Sect.~\ref{sec:parallel_transp}.
Algorithm~\ref{algo:RCG} provides a pseudocode for the RCG method.

\begin{algorithm}
\SetAlgoLined
 Given objective function $E$, power manifold $\powerTorus$, initial iterate $\f^{(0)}\in \powerTorus$, projector $ \P_{\f} $ from $\R^{n\times 3}$ to $\mathrm{T}_{\f}\powerTorus$, retraction $\Retraction_{\f}$ from $\mathrm{T}_{\f}\powerTorus$ to $\powerTorus$\;
 \KwResult{Sequence of iterates $\lbrace f^{(k)} \rbrace$.}
 $k \leftarrow 0$\;
Compute the Euclidean gradient of the objective function $\nabla E(f^{(k)})$\;
Compute the Riemannian gradient as $ \grad E(f^{(k)}) = \P_{\f^{(k)}} \! \big( \nabla E(f^{(k)}) \big)$ using Algorithm~\ref{algo:proj_onto_tg_space}\;
Set $ \boldd^{(k)} =  -\grad E(f^{(k)}) $\;
 \While{$f^{(k)}$ does not sufficiently minimizes $E$}{
     Compute a step size $\alpha_{k} > 0$ with a line-search procedure that satisfies the sufficient decrease condition\;
     Set $\f^{(k+1)} = \Retraction_{\f^{(k)}}(\alpha_{k} \boldd^{(k)})$\;
     Compute the Euclidean gradient of the objective function $\nabla E(f^{(k+1)})$ \;
     Compute the Riemannian gradient as $ \grad E(f^{(k+1)}) = \P_{\f^{(k+1)}}\! \big( \nabla E(f^{(k+1)}) \big)$ using Algorithm~\ref{algo:proj_onto_tg_space}\;
     Compute the parallel transport $ \mathcal{T}_{\f^{(k)}\to \f^{(k+1)}} $ of $ \boldd^{(k)} $ using Algorithm~\ref{algo:parallel_transport}\;
     Compute the scalar $\beta_{k}^{\mathrm{FR}} = \frac{\|\grad E(f^{(k+1)})\|^{2}}{\|\grad E(f^{(k)})\|^{2}} $\;
     Set the new direction $  \boldd^{(k+1)} = -\grad E(f^{(k+1)}) + \beta_{k}^{\mathrm{FR}} \, \mathcal{T}_{\f^{(k)}\to \f^{(k+1)}} \, \boldd^{(k)} $\;
     $ k \leftarrow k + 1 $\;
 }
 \caption{The Riemannian conjugate gradient method on $\powerTorus$.}\label{algo:RCG}
\end{algorithm}

The computation of $ \beta_{k}^{\mathrm{FR}} $ is crucial for the performance of the method.
The above ratio~\eqref{eq:fletcher_reeves} is due to Fletcher and Reeves~\protect{\cite{Fletcher:1964}} and it is the one that we use in this work. Other choices are possible, like those due to Polak and Ribière~\protect{\cite{Polak:1969}} or Hestenes and Stiefel~\protect{\cite{Hestenes:1952}}. We also experimented with these rules, but the experiments show that while keeping all other parameters the same, the Fletcher--Reeves update rule provided the best results. 

Convergence results for the Riemannian conjugate gradient methods based on the Fletcher--Reeves ratio~\eqref{eq:fletcher_reeves} are discussed in~\protect{\cite[\S 6.1.1]{Sato:2022}} and~\protect{\cite[\S 4.4]{Sato:2021}}.

\subsection{Projected conjugate gradient method}
The nonlinear conjugate gradient method for solving optimization problems was introduced by Fletcher and Reeves in~\protect{\cite{Fletcher:1964}}.
A pseudocode of the conjugate gradient method with the Fletcher--Reeves update rule is given, e.g., in~\protect{\cite[p.~121]{NW:2006}}, from which we adapted our pseudocode for the projected conjugate gradient (PCG) method; see Algorithm~\ref{algo:PCG} below. The main difference with respect to the CG method is the additional projection step onto the torus as described in Algorithm~\ref{algo:projection_on_torus}.

\begin{algorithm}
\SetAlgoLined
 Given objective function $E$, power manifold $\powerTorus$, initial iterate $\f^{(0)}\in \powerTorus$, projector $ \Pi $ from $\R^{n\times 3}$ to $\powerTorus$\;
 $k \leftarrow 0$\;
 Compute the Euclidean gradient of the objective function $\nabla E_{k} \equiv \nabla E(f^{(k)})$\;
 Set  $ \boldd^{(k)} = -\nabla E(f^{(k)}) $\;
 \While{$\f^{(k)}$ does not sufficiently minimizes $E$}{
     Compute a step size $\alpha_{k} > 0$ with a line-search procedure that satisfies the sufficient decrease condition\;
     Set $\f^{(k+1)} = \Pi_{\powerTorus}\left( \f^{(k)} + \alpha_{k} \, \boldd^{(k)} \right)$\;
     Compute the Euclidean gradient of the objective function $\nabla E_{k+1} \coloneqq \nabla E(\f^{(k+1)})$\;
     $ \beta_{k+1}^{\mathrm{FR}} \leftarrow \frac{\nabla E_{k+1}\tr\nabla E_{k+1}}{\nabla E_{k}\tr\nabla E_{k}} $\;
     Set the anti-gradient direction $ \boldd^{(k+1)} \leftarrow -\nabla E(\f^{(k+1)}) + \beta_{k+1}^{\mathrm{FR}} \boldd^{(k)} $\;
     $ k \leftarrow k + 1 $\;
 }
 \caption{The projected conjugate gradient method on $\powerTorus$.}\label{algo:PCG}
\end{algorithm}

We may regard the PCG method as an approximate version of the RCG method. The main differences between the two algorithms are that:
\begin{itemize}
    \item PCG uses the Euclidean gradient, not the Riemannian gradient.
    \item The line-search procedure is performed in the embedding space $ \R^{n \times 3} $, not on the tangent space.
    \item The new iterate is obtained by projecting directly onto $\powerTorus$; there is no intermediate step involving the tangent space.
    \item The scalar $ \beta_{k+1}^{\mathrm{FR}} $ is computed from Euclidean gradients, not from Riemannian gradients.
\end{itemize}
The numerical experiments of Sect.~\ref{sec:numerical_experiments} show that, in practice, the PGM always converges.

\section{Numerical experiments} \label{sec:numerical_experiments}

In this section, we report and discuss the numerical results for genus-one area-preserving mapping for the mesh models shown in Figure~\ref{fig:All_mesh_models}. The mesh models are obtained from several online sources, among them {\em 180 Wrapped Tubes}\footnote{\url{https://pub.ista.ac.at/~edels/Tubes}}, Keenan's 3D Model Repository\footnote{\url{https://www.cs.cmu.edu/~kmcrane/Projects/ModelRepository}}, and CGTrader\footnote{\url{https://www.cgtrader.com}}.

To obtain more uniform meshes, some models were remeshed using the {\em JIGSAW} mesh generators \protect{\cite{Engw14,EnIv14,Engw15,EnIv16,Engw16}}.
In Figure~\ref{fig:All_mesh_models}, the mesh models are ordered according to increasing number of faces and vertices, from left to right and top to bottom.

\begin{figure}[htbp]
   \centering
   \resizebox{\textwidth}{!}{
      \begin{tabular}{lcccc}
         \hline\\[-0.2cm]
         Model Name  & Knot & Triangle Torus 7/3 & Triangle Torus 10/3 & Bob \\ 
         \# Faces    & 2,880 & 3,360 & 3,500 & 6,174 \\ 
         \# Vertices & 1,440 & 1,680 & 1,750 & 3,087 \\[0.25cm]
         & \includegraphics[height=3cm]{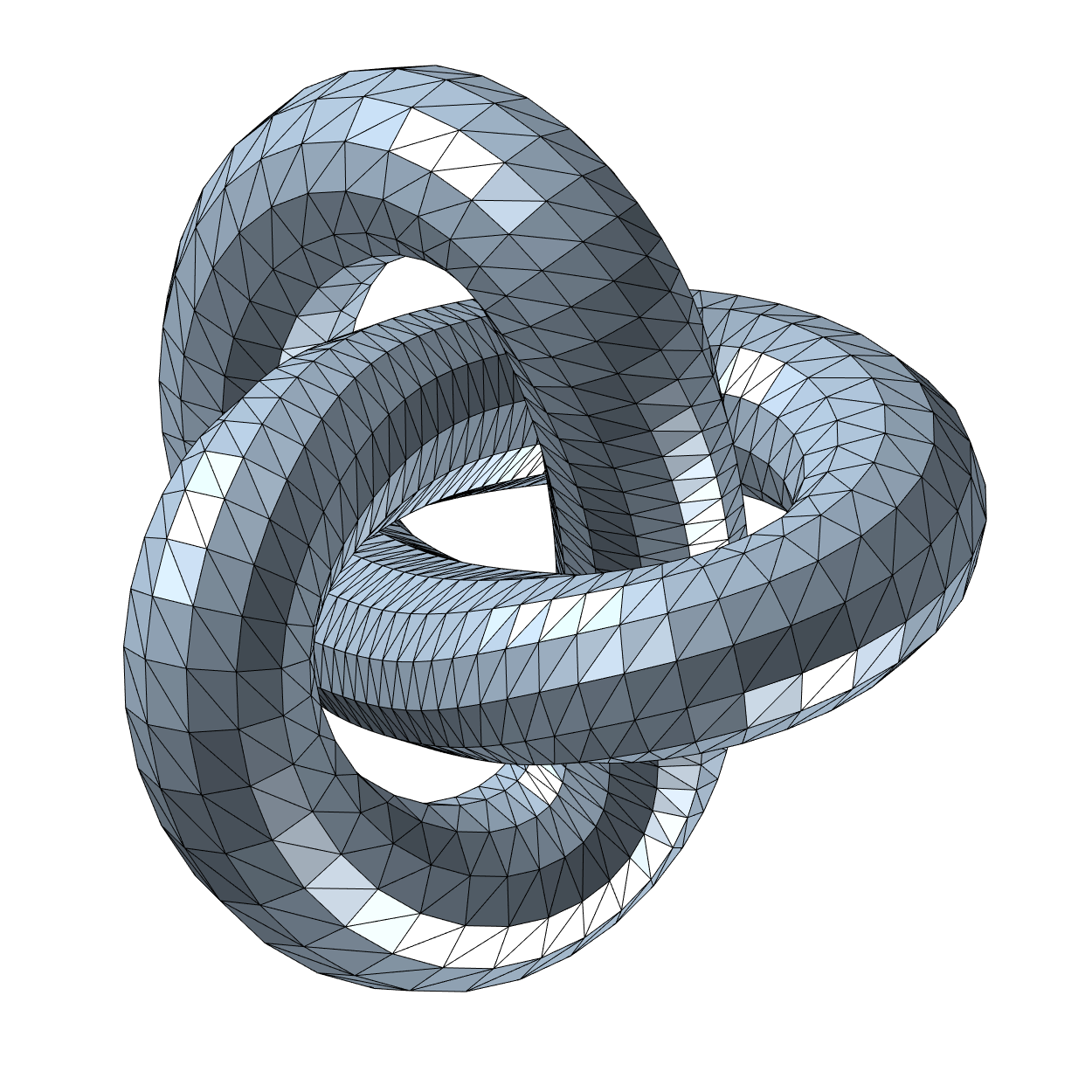} &
         \includegraphics[height=3cm]{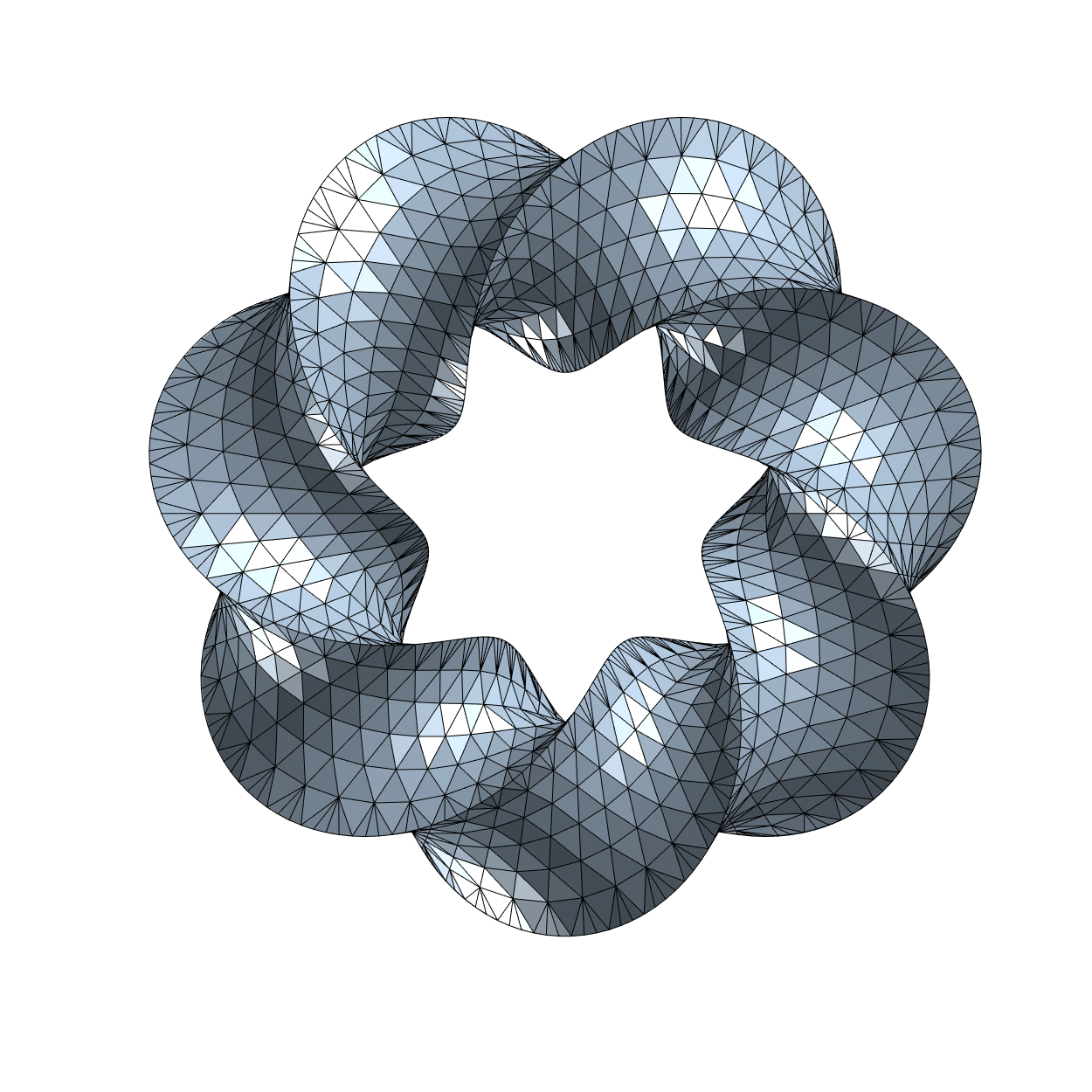} & \includegraphics[height=3cm]{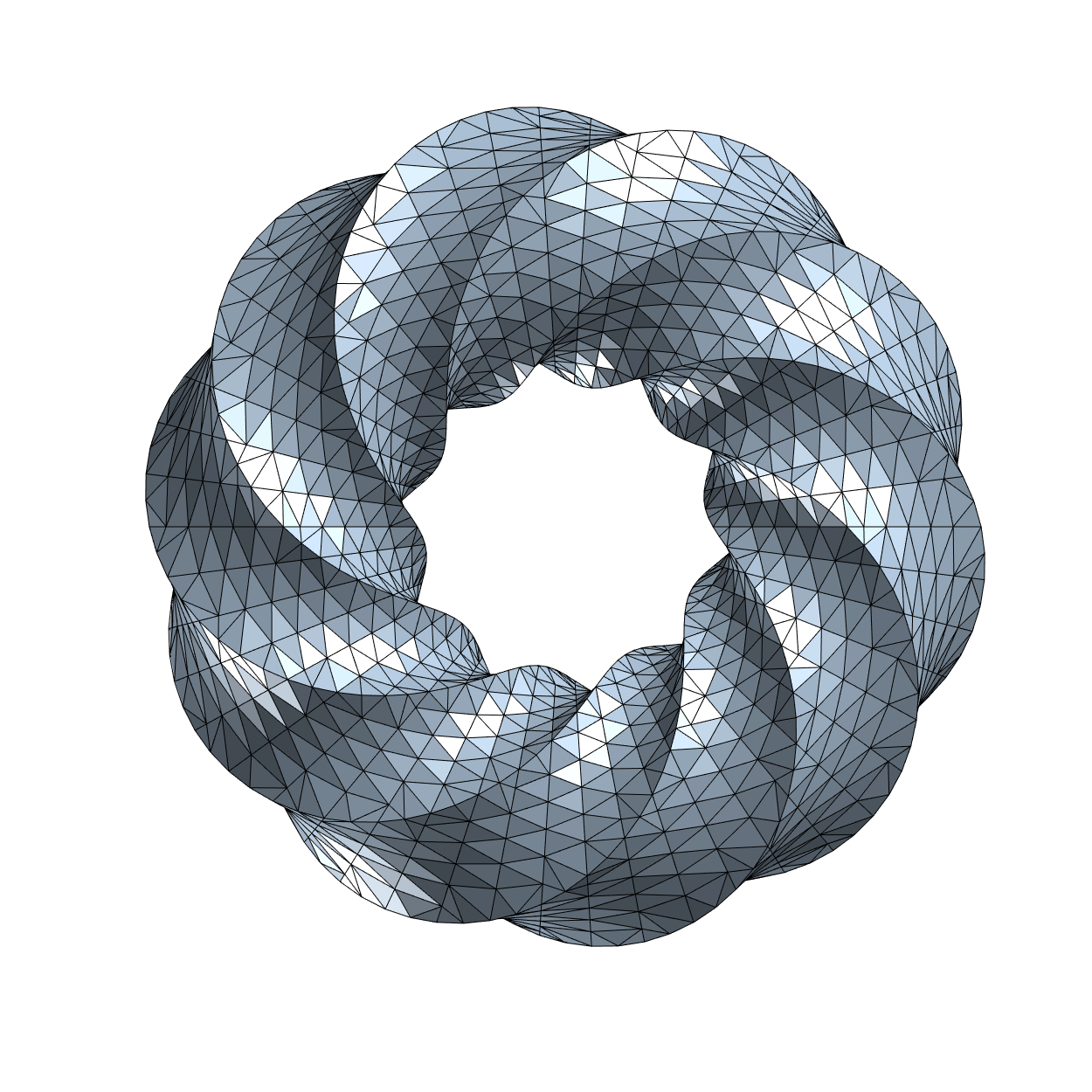} &  \includegraphics[height=3cm]{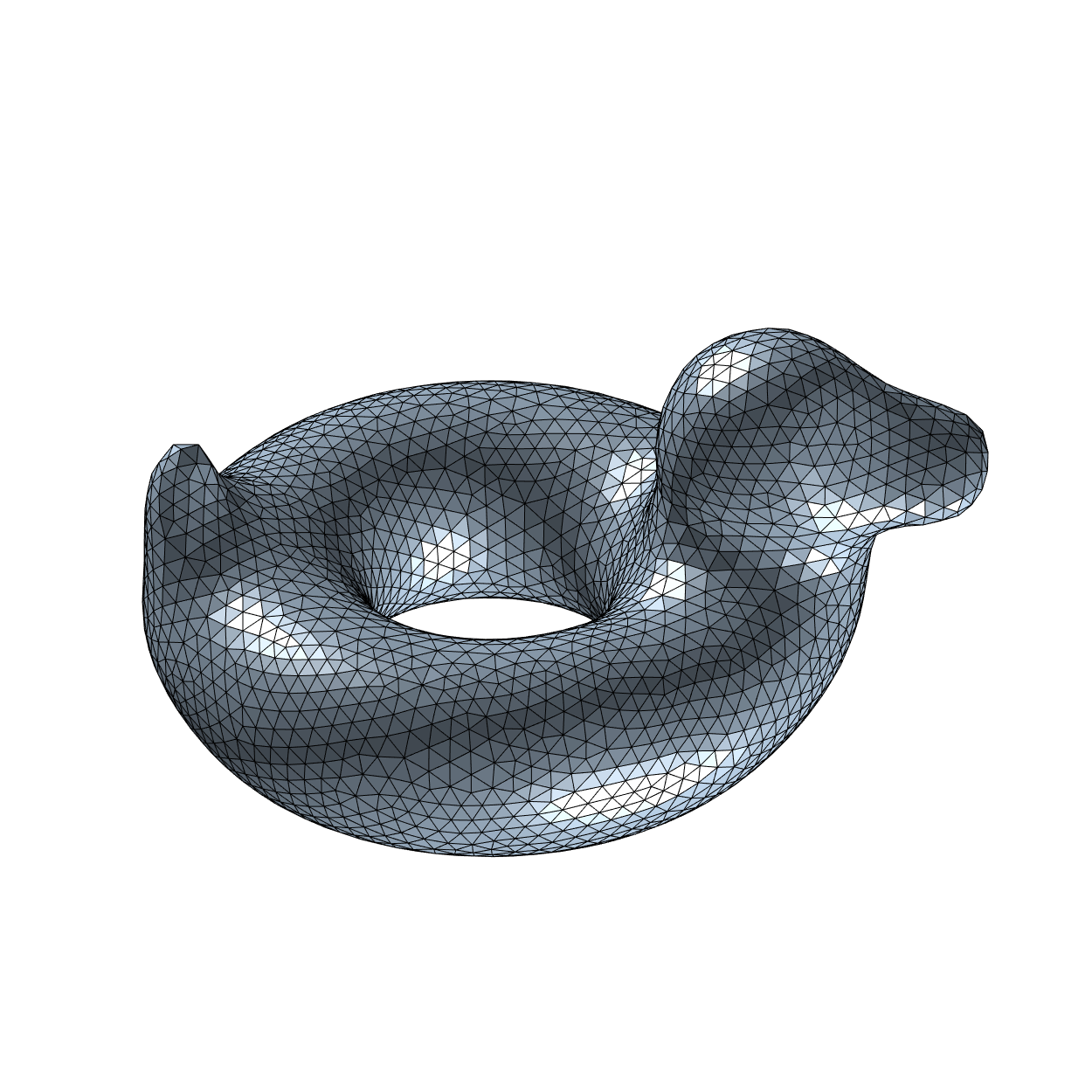} \\ \hline
        \\\hline\\[-0.2cm]
        Model Name  & Square Knot (1, 8, 0) & Circle Knot (3, 5/3) & Vertebrae & Kitten \\ 
        \# Faces    & 11,200 & 12,000 & 16,420 & 20,000 \\ 
        \# Vertices &  5,600 &  6,000 &  8,210 & 10,000 \\[0.25cm]
        & \includegraphics[height=3cm]{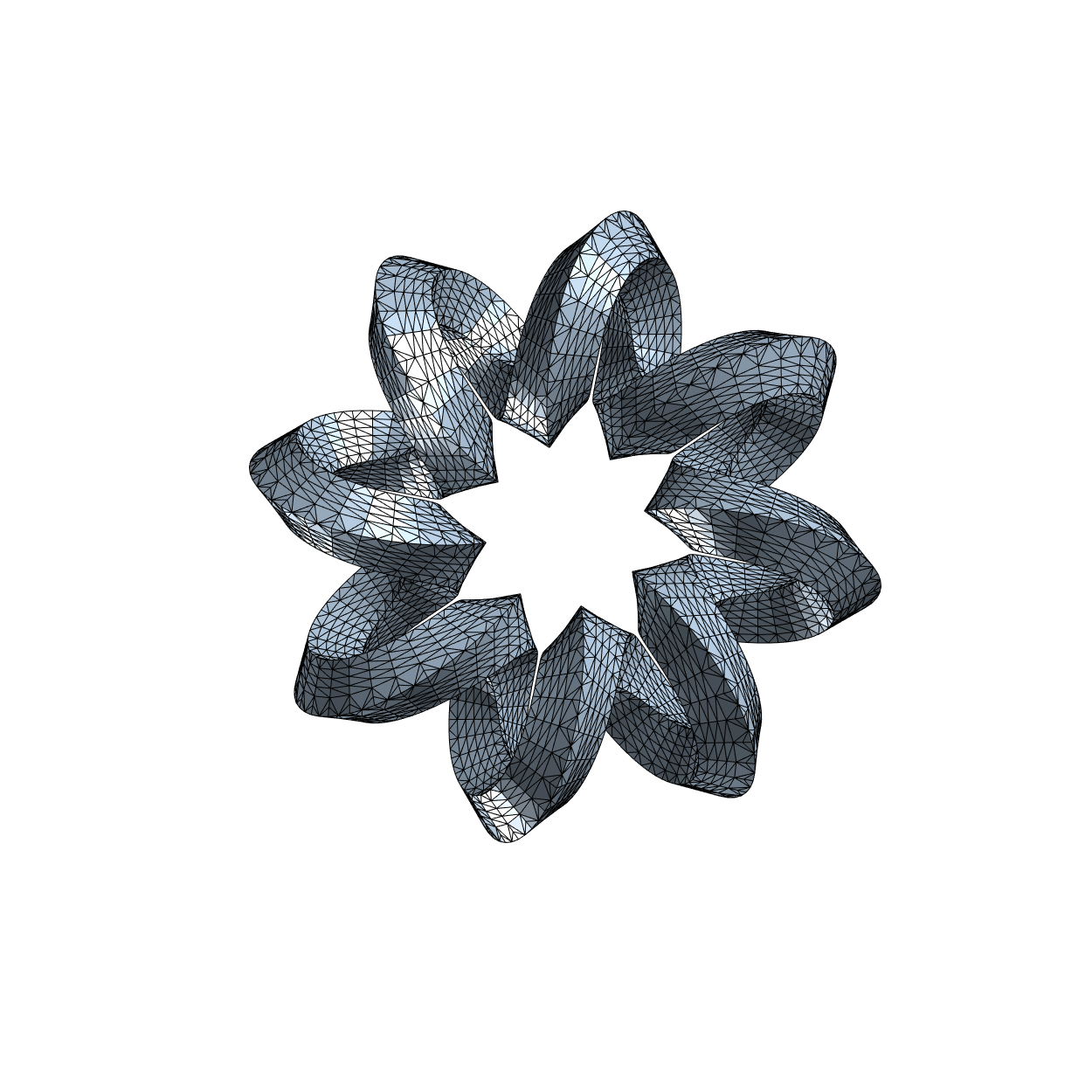} & \includegraphics[height=3cm]{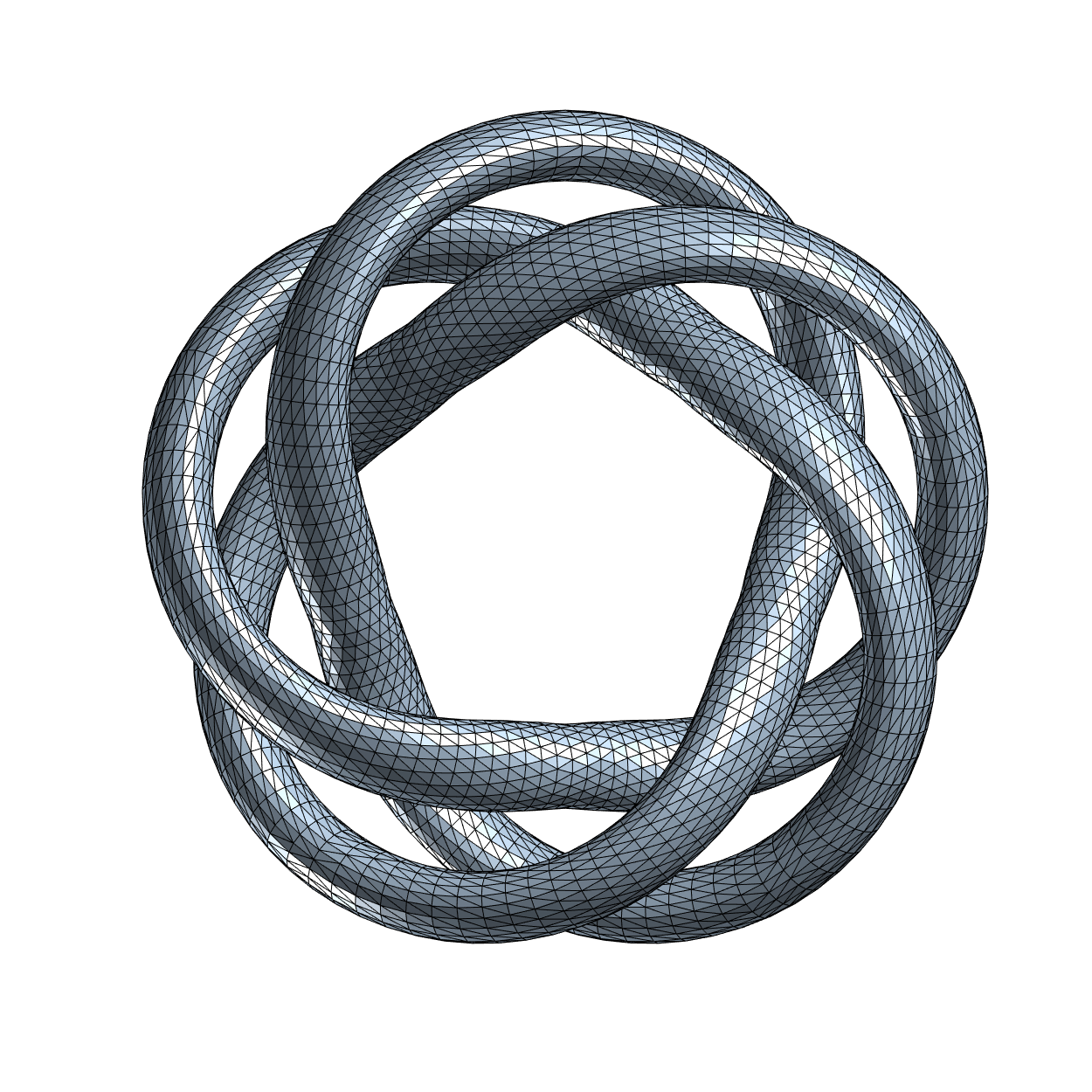} & \includegraphics[height=3cm]{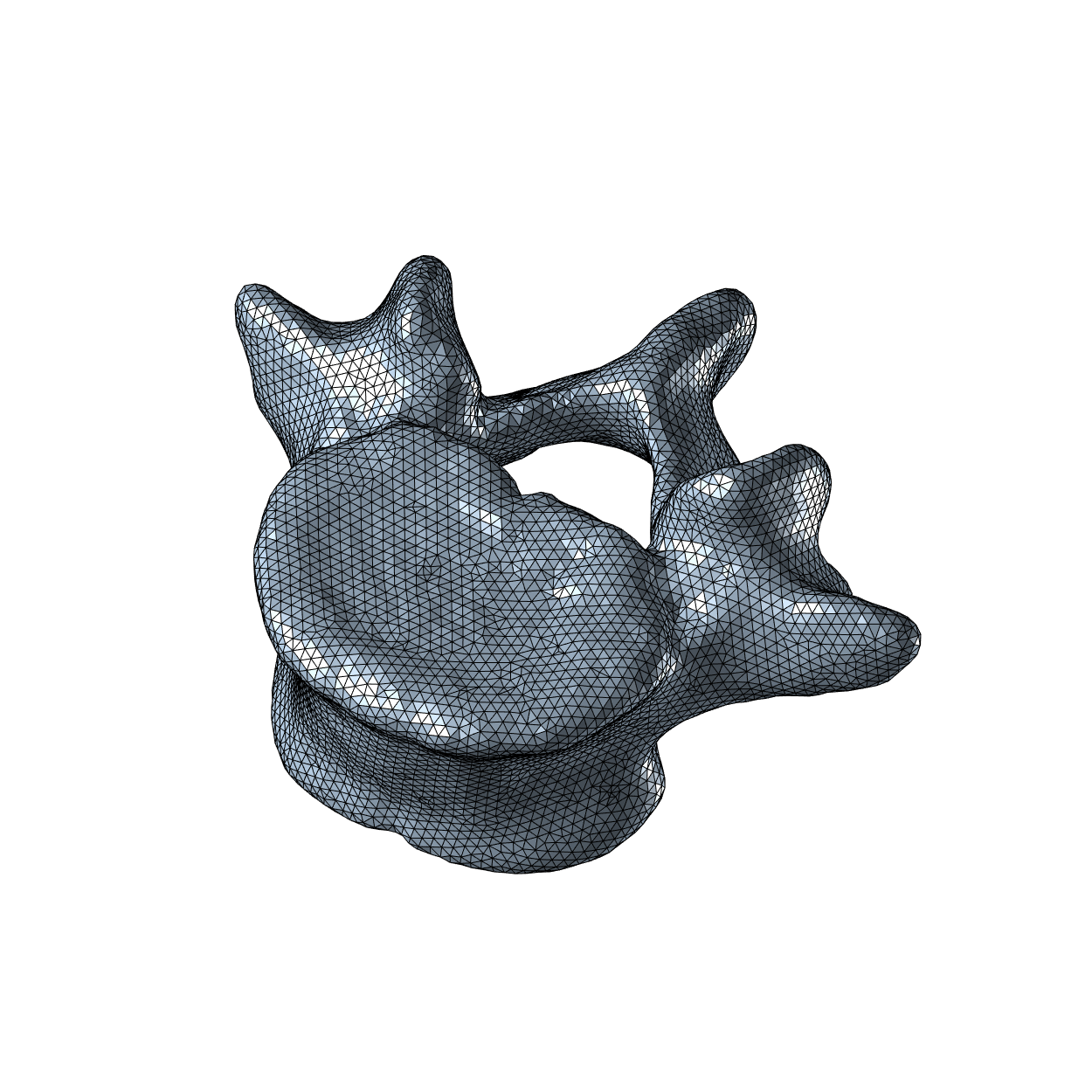} & \includegraphics[height=3cm]{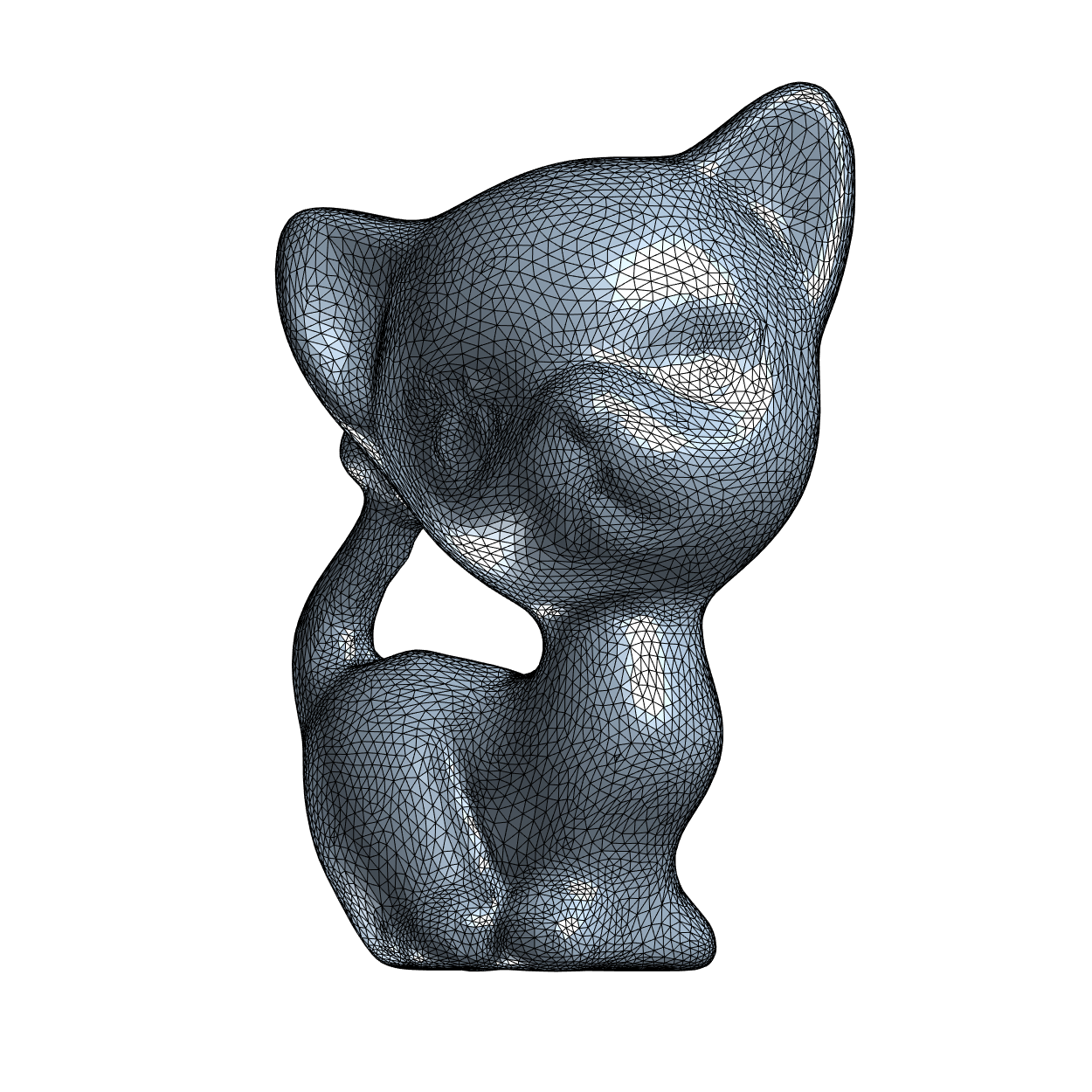}  \\ \hline
        \\\hline\\[-0.2cm]
        Model Name  & Cogwheel & Chess Horse & Rusted Gear & Meander Ring \\
        \# Faces    &  27,228 & 46,016 & 50,150 & 63,208 \\ 
        \# Vertices &  13,614 & 23,008 & 25,075 & 31,604  \\[0.25cm] & \includegraphics[height=3cm]{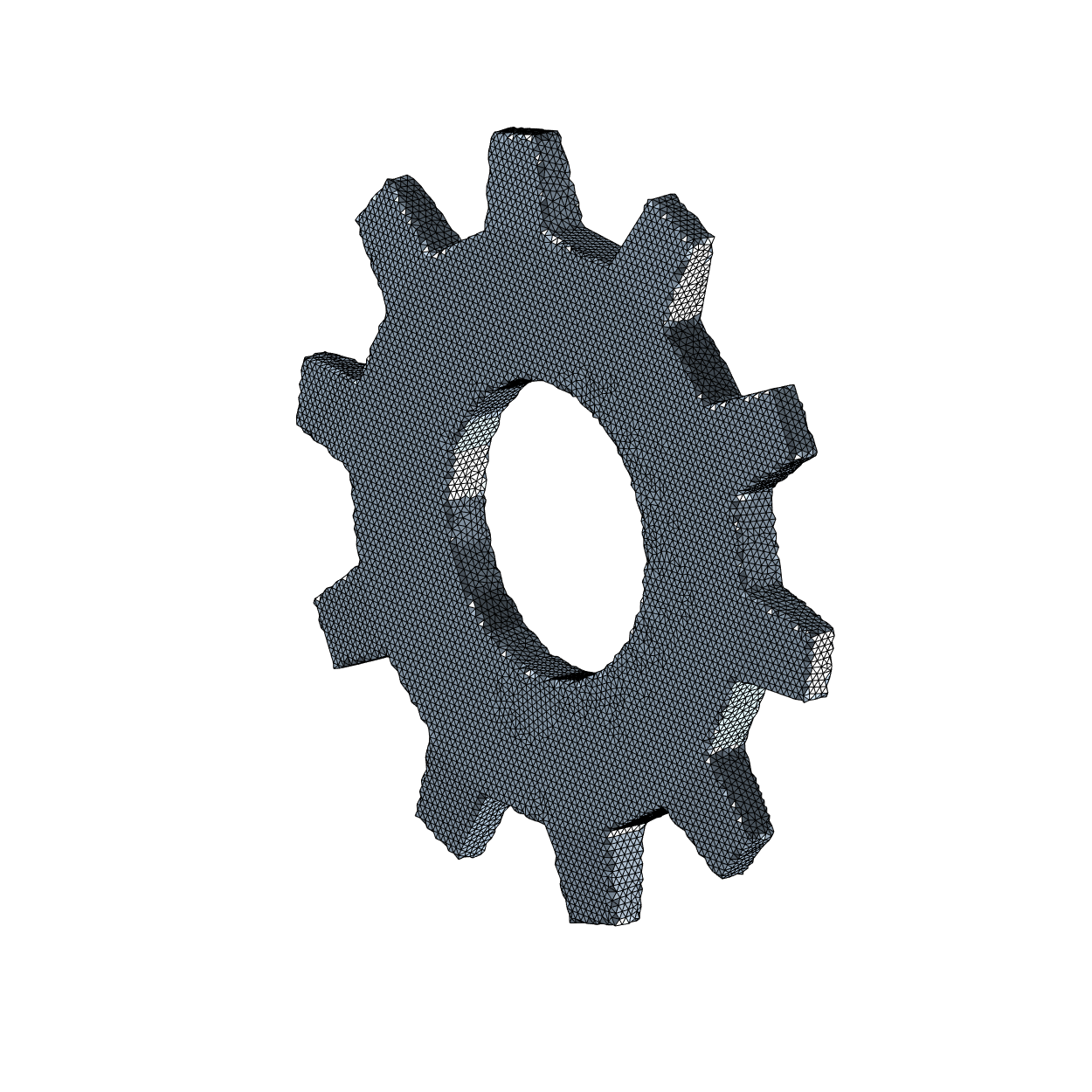} & \includegraphics[height=3cm]{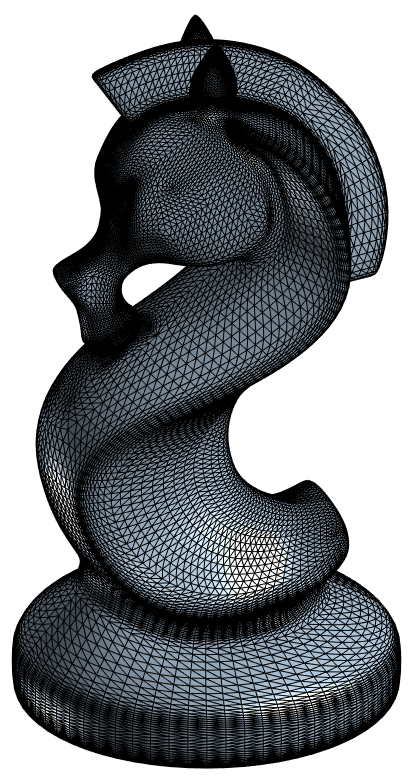} & \includegraphics[height=3cm]{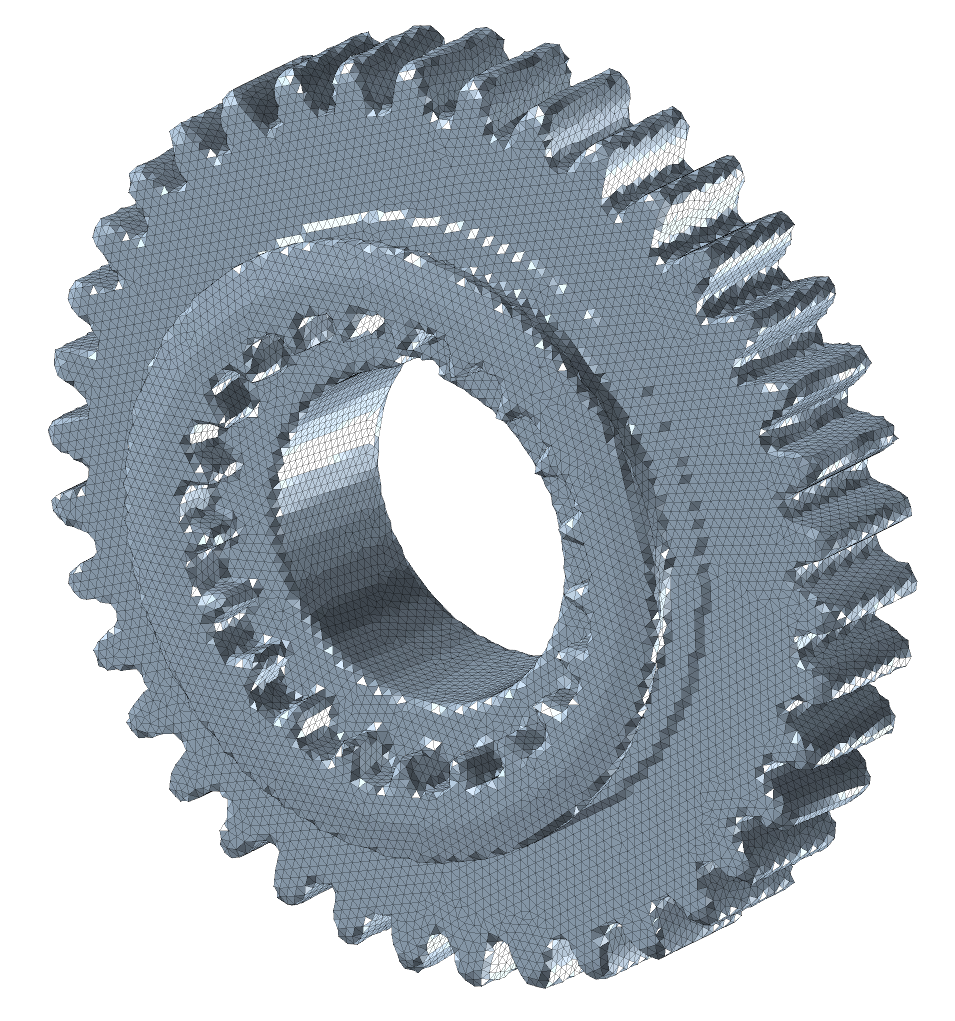} & \includegraphics[height=3cm]{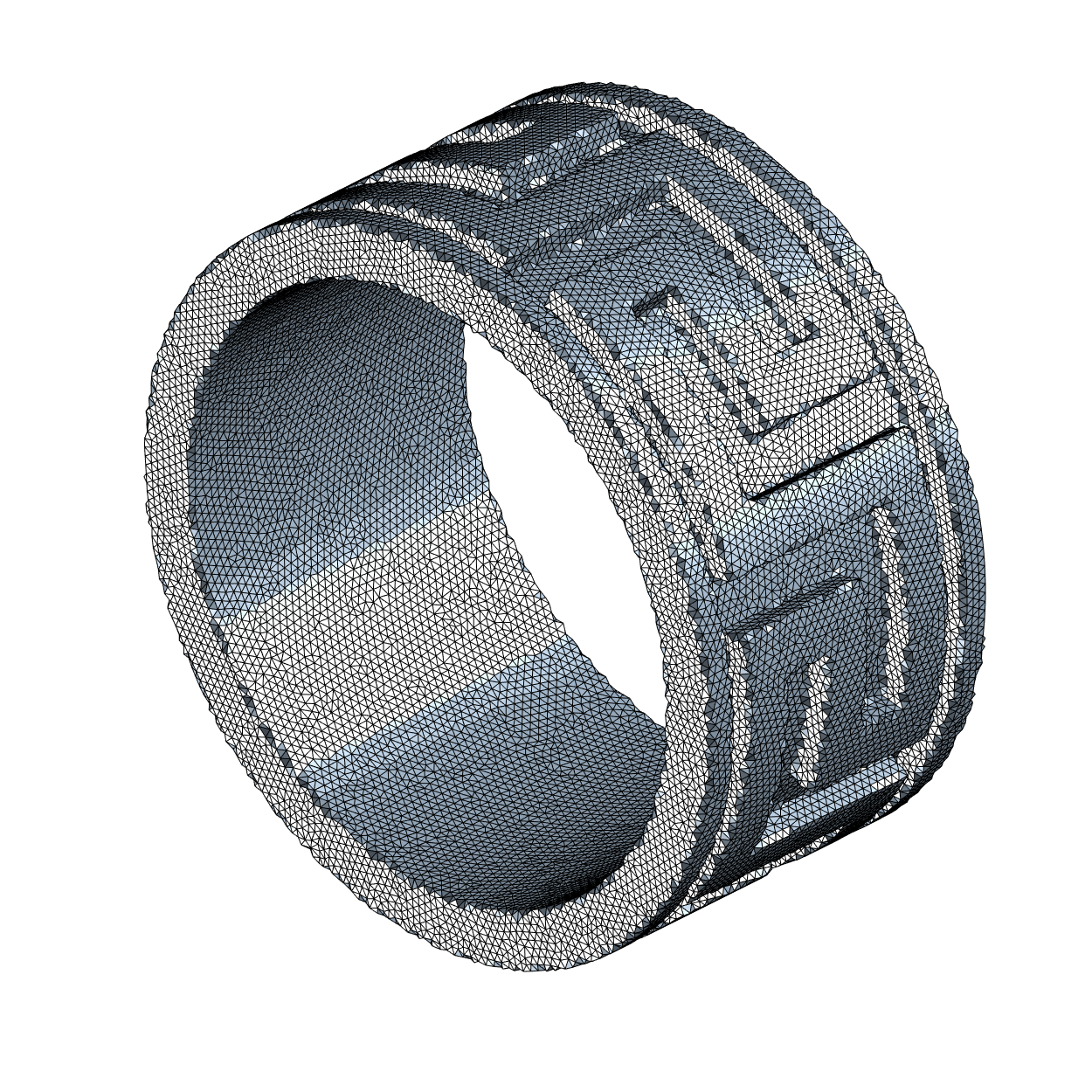}  \\ \hline
        \\\hline\\[-0.2cm]
        Model Name  & Lumbar 2 & Lumbar 4 & Thoracic 10 & Thoracic 12 \\
        \# Faces    & 300,000 & 300,000 & 300,000 & 300,040 \\ 
        \# Vertices & 150,000 & 150,000 & 150,000 & 150,020 \\[0.25cm] & \includegraphics[height=4cm]{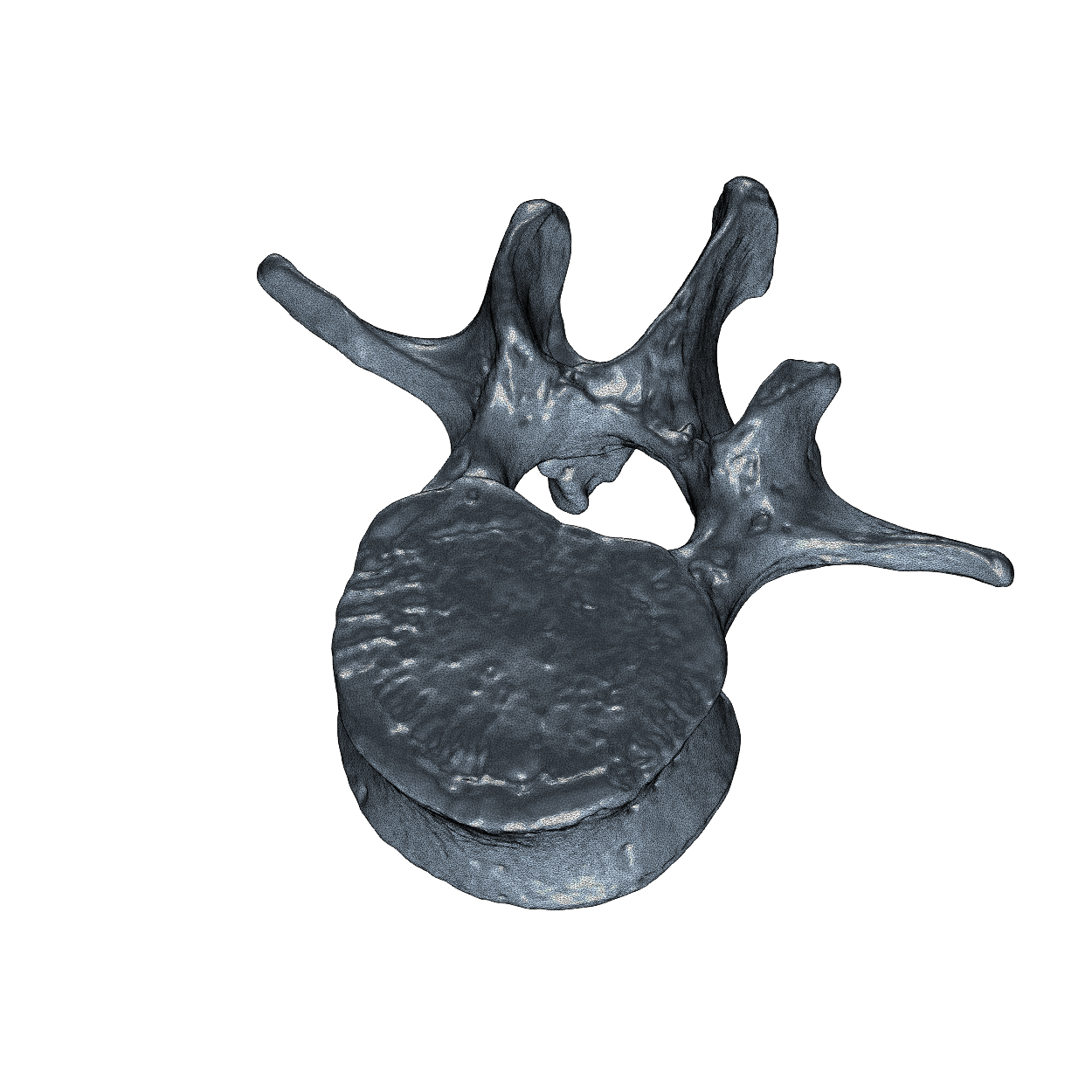} & \includegraphics[height=4cm]{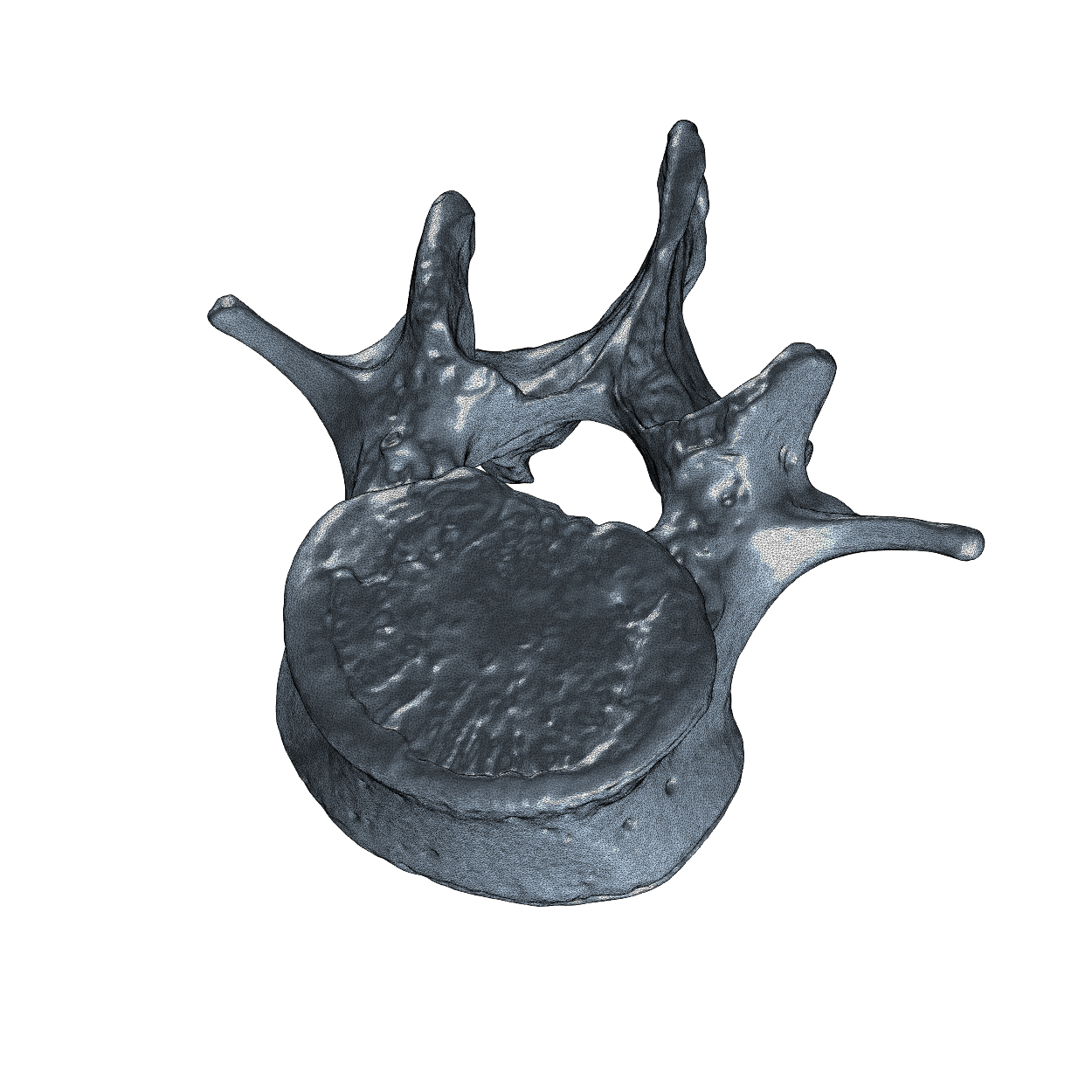} & \includegraphics[height=4cm]{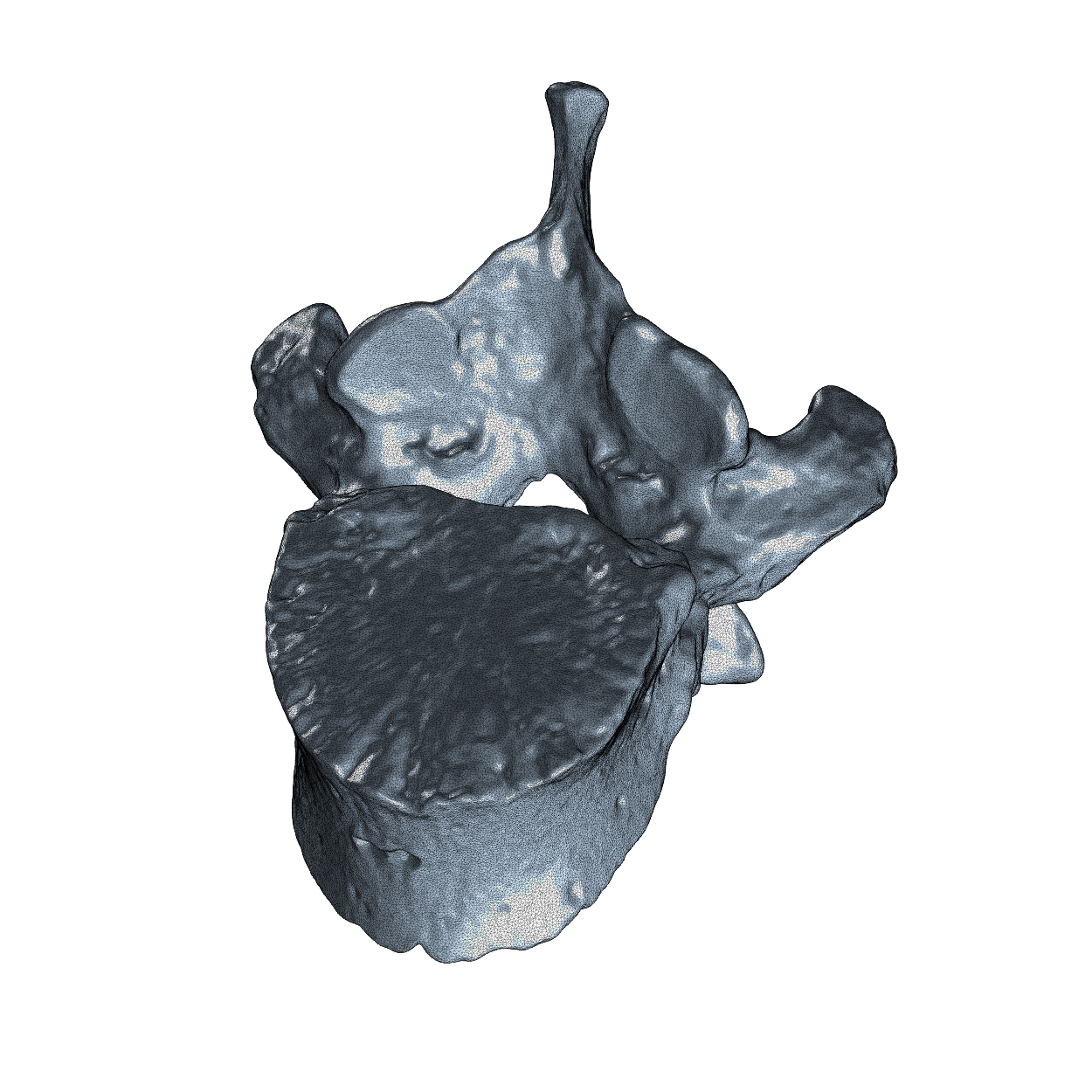} & \includegraphics[height=4cm]{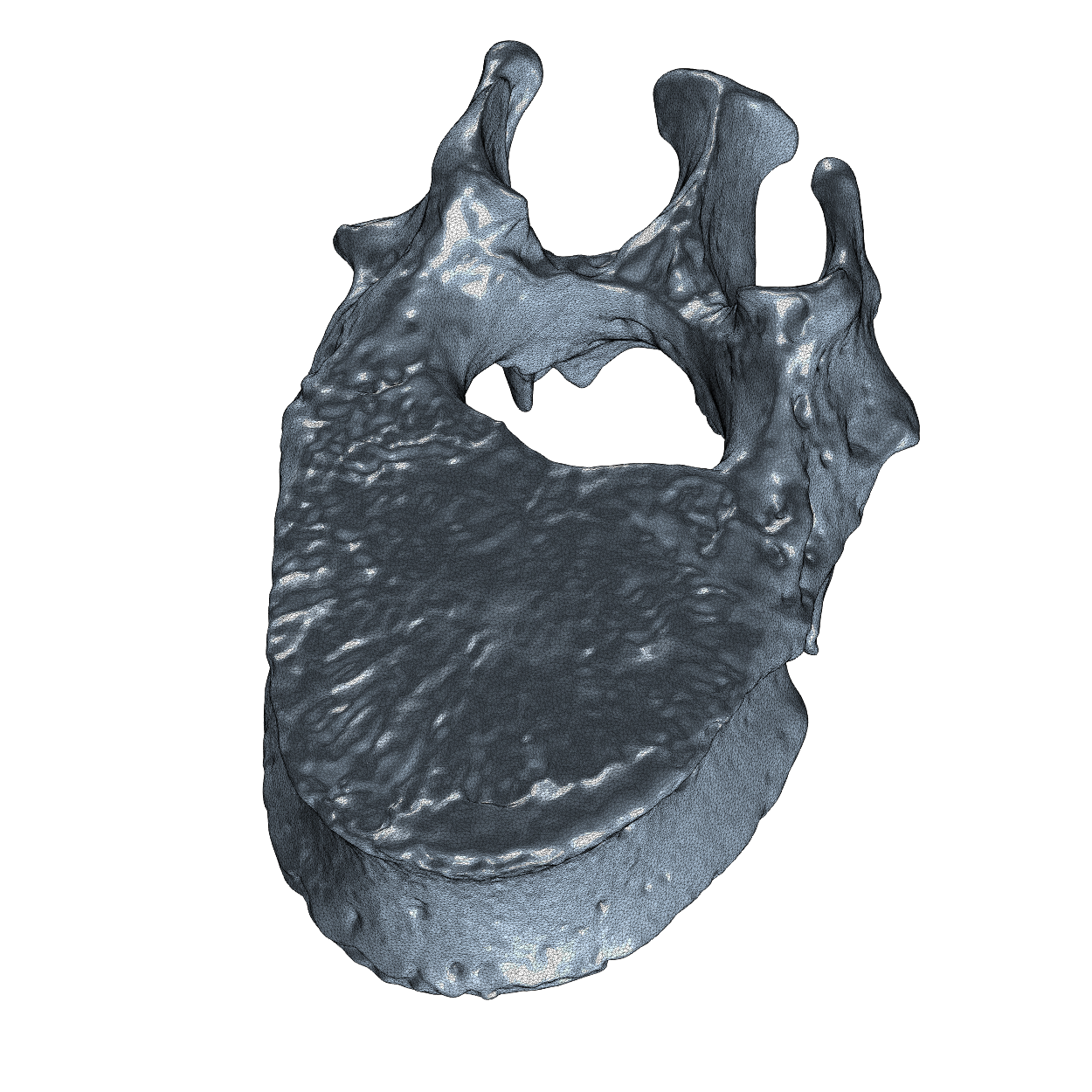} \\
        \hline
    \end{tabular}
}
\caption{The benchmark mesh models used in this paper.}
\label{fig:All_mesh_models}
\end{figure}

\begin{figure}[htbp!]
\centering
    \begin{minipage}[t]{.425\textwidth}
        \vspace{0pt}
        \subfigure[]{\includegraphics[width=\textwidth]{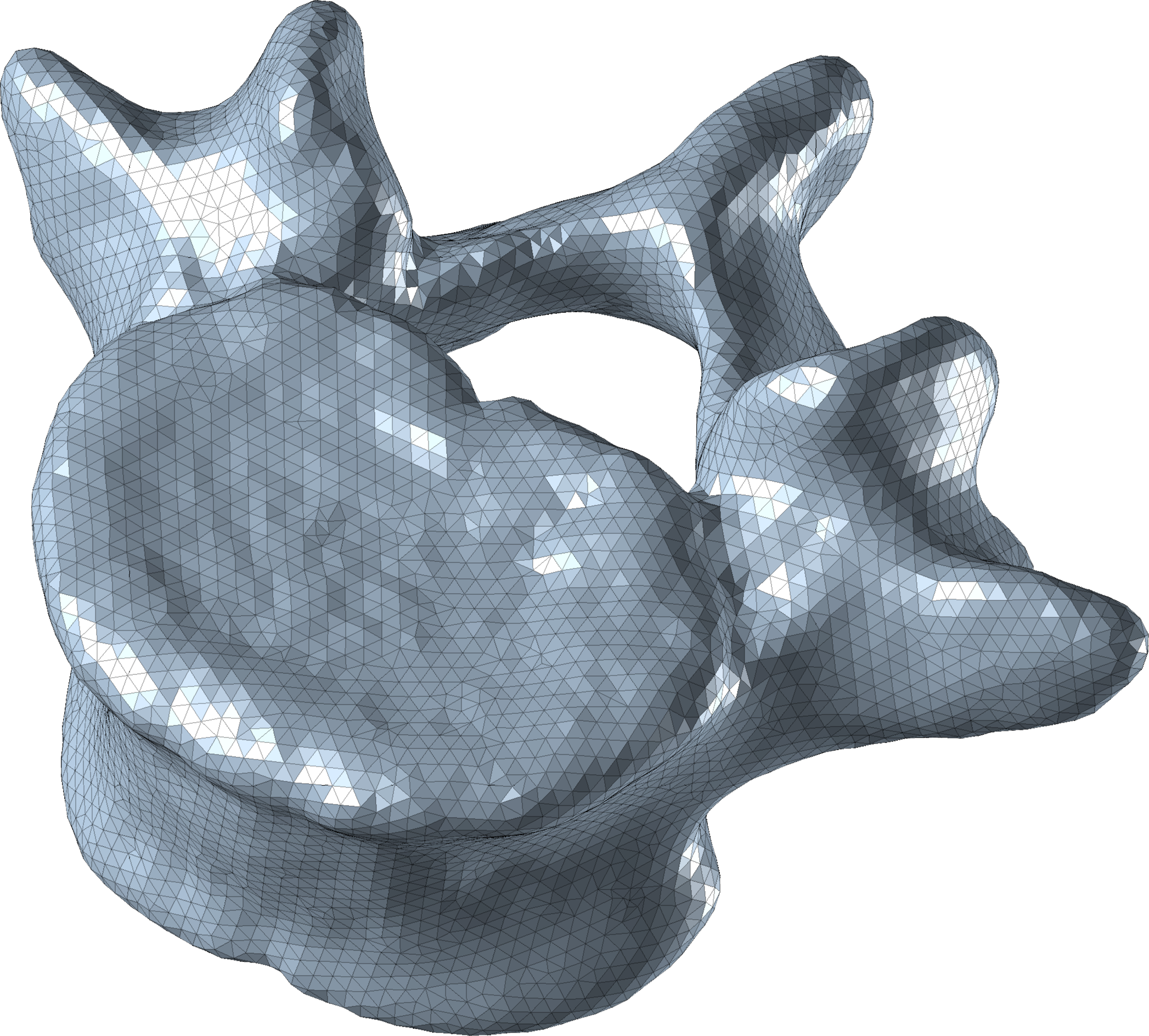}}\label{fig:vertebra_a}\par
        \hfill
    \end{minipage}\hfill
    \begin{minipage}[t]{.525\textwidth}
        \vspace{0pt}
        \subfigure[]{\includegraphics[width=\textwidth]{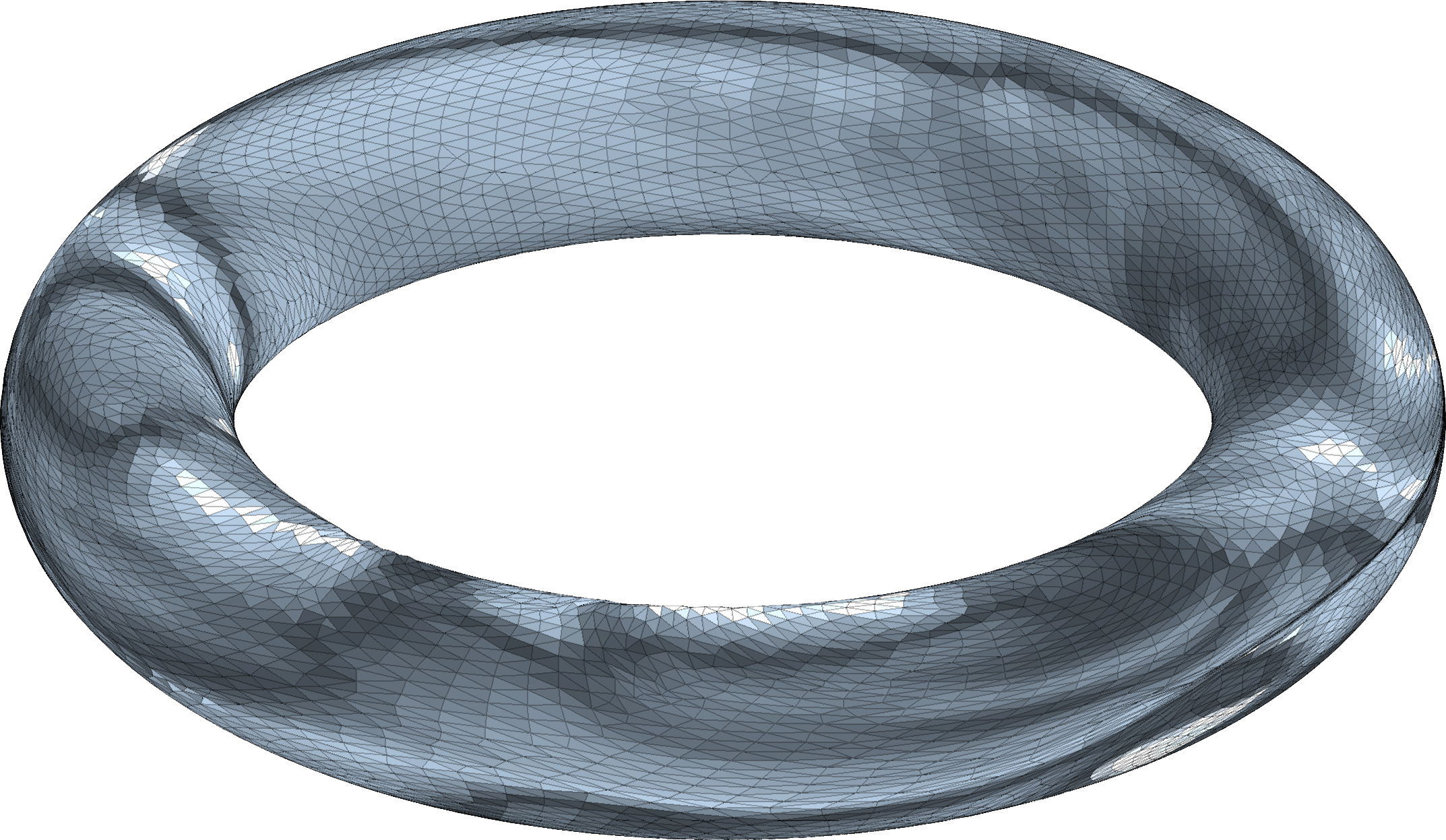}}\label{fig:vertebra_b}
        \vspace{0pt}
    \end{minipage}\\
\subfigure[]{\includegraphics[width=\textwidth]{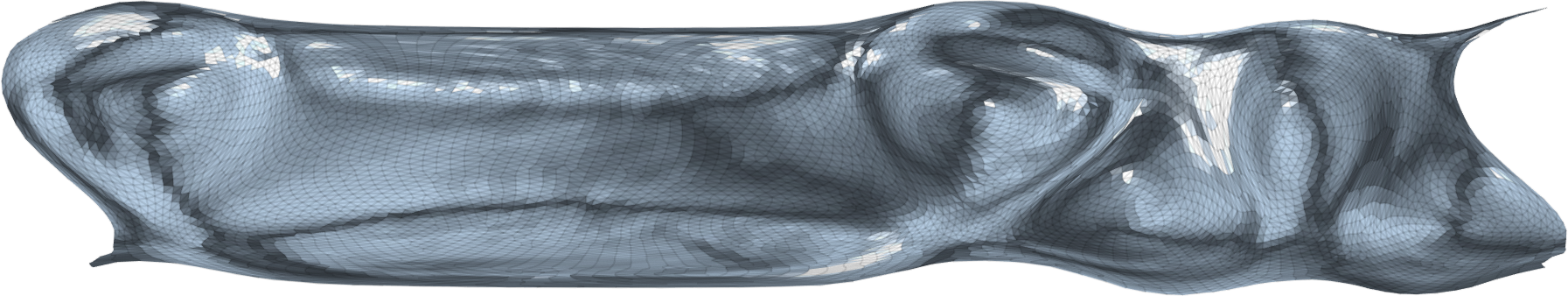}}\label{fig:vertebra_c}
\caption{The simplicial surface Vertebrae (a), its area-preserving parameterization on $\Torus$ (b), and the associated fundamental domain (c).}
\label{fig:vertebra_and_torus}
\end{figure}

We performed numerical experiments for all the four algorithms presented in Sect.~\ref{sec:riemannian_algorithms}, namely: projected gradient descent, projected conjugate gradient, Riemannian gradient method, and Riemannian conjugate gradient method. In all these methods, we adopted the line-search strategy that uses the quadratic/cubic interpolant of \protect{\cite[\S 6.3.2]{DennisSchnabel:1996}}, described in Appendix~\ref{sec:line_search}. 
In all the experiments, we monitor the energy $E$ defined in~\eqref{eq:objective_function}, the number of folding triangles (denoted by \#Fs), and the ratio between the standard deviation (SD) and the mean of the local area ratios. This quantity has been considered in~\protect{\cite{CR:2018,Sutti_Yueh:2024}}. If necessary, after each method, a bijectivity correction is applied to get rid of the (potential) overlapping triangles. We refer the reader to~\protect{\cite[\S 5.2]{Sutti_Yueh:2024}} for a description of the bijectivity correction procedure.

Table~\ref{tab:table_1} is for the projected gradient and the projected conjugate gradient methods, while Table~\ref{tab:table_2} is for their Riemannian counterparts. We point out that, after applying the bijectivity correction, the number of folding triangles is zero for all the models; hence, the column \#Fs appears only once in the tables.

\begin{table}[htbp]
   \caption{Comparison between the numerical results for the projected gradient and the projected conjugate gradient methods.} 
   \label{tab:table_1}
   \center
   \centering
   \resizebox{\textwidth}{!}{
      \begin{tabular}{@{} *{11}{c}}
            \toprule
            & \multicolumn{5}{c}{Projected Gradient Method} & \multicolumn{5}{c}{Projected Conjugate Gradient Method} \\ 
            \cmidrule(lr){2-6} \cmidrule(lr){7-11}
            & \multicolumn{3}{c}{Before bij. correction} & \multicolumn{2}{c}{After bij. correction} & \multicolumn{3}{c}{Before bij. correction} & \multicolumn{2}{c}{After bij. correction} \\
            \cmidrule(lr){2-6} \cmidrule(lr){7-11}
Model Name & SD/Mean & $E(f)$ & \#Fs & SD/Mean & $E(f)$ & SD/Mean & $E(f)$ & \#Fs & SD/Mean & $E(f)$ \\
            \cmidrule(r){1-1} \cmidrule(lr){2-6} \cmidrule(lr){7-11}
Knot & 0.0667 & $ 4.49 \times 10^{-2} $ & 0 & --- & --- & \textbf{0.0406} & $ 1.60 \times 10^{-2}$ & 0 & --- & --- \\
Triangle Torus 7/3 & 0.0849 & $ 3.06\times 10^{-2} $ & 70 & 0.1620 & $ 7.59\times 10^{-2} $ & 0.0446
 & $ 8.69 \times 10^{-3}$ & 14 & 0.0711 & $ 1.66\times 10^{-2} $ \\
Triangle Torus 10/3 & 0.0969 & $ 4.34\times 10^{-2} $ & 195 & 0.2628 & $ 2.21\times 10^{-1} $ & 0.0421
 & $ 7.82 \times 10^{-3}$ & 204 & \textbf{0.2052} & $ 1.23\times 10^{-1} $ \\
Bob Isotropic & 0.2633 & $ 3.77 \times 10^{-1} $ & 186 & 0.3640 & $ 6.69 \times 10^{-1} $ & 0.1600 & $ 1.46 \times 10^{-1}$ & 342 & \textbf{0.3318} & $ 5.55 \times 10^{-1}$ \\
Square Knot (1, 8, 0) & 0.0306 & $ 2.16 \times 10^{-2} $ & 0 & --- & --- & \textbf{0.0246}
 & $ 1.38 \times 10^{-2}$ & 0 & --- & --- \\
Circle Knot (3, 5/3) & 0.0154 & $ 5.63 \times 10^{-3} $ & 0 & --- & --- & \textbf{0.0106} & $ 2.65 \times 10^{-3}$ & 0 & --- & --- \\
Vertebrae & 0.3141 & $ 5.35 \times 10^{-1} $ & 287 & 0.3457 & $ 6.31 \times 10^{-1} $ & 0.2399 & $ 3.20 \times 10^{-1}$ & 525 & \textbf{0.3109} & $ 5.13 \times 10^{-1}$ \\
Kitten & 0.5133 & $ 1.38 \times 10^{0} $ & 203 & 0.5319 & $ 1.46 \times 10^{0} $ & 0.4123 & $ 8.92 \times 10^{-1}$ & 433 & 0.4412 & $ 9.84 \times 10^{-1}$ \\
Cogwheel & 0.2434 & $ 3.68 \times 10^{-1} $ & 88 & 0.2514 & $ 3.87 \times 10^{-1} $ & 0.0782 & $ 3.86 \times 10^{-2}$ & 161 & \textbf{0.1109} & $ 7.35 \times 10^{-2}$ \\
Chess Horse & 0.8970 & $ 6.02 \times 10^{0} $ & 241 & 0.9024 & $ 6.08 \times 10^{0} $ & 0.8099 & $ 5.26 \times 10^{0}$ & 441 & \textbf{0.8070} & $ 5.29 \times 10^{0}$ \\
Rusted Gear 
& 0.5923 & $ 1.43 \times 10^{0} $ & 641 & 0.6398 & $ 1.61 \times 10^{0} $ & 0.4192 & $ 7.48 \times 10^{-1}$ & 1651 & \textbf{0.5072} & $ 1.02 \times 10^{0}$ \\
Meander Ring 
& 0.4070 & $ 6.30 \times 10^{-1} $ & 0 & --- & --- & 0.0233 & $ 2.00 \times 10^{-3}$ & 5 & \textbf{0.0256} & $ 2.34 \times 10^{-3}$ \\
Lumbar 2    & 0.7433 & $ 1.45 \times 10^{0} $  &  267 & 0.7572 & $ 1.45 \times 10^{0} $  & 0.5667 & $ 1.36 \times 10^{0}$  & 3305 & 0.5760 & $ 1.44 \times 10^{0}$ \\
Lumbar 4    & 0.5000 & $ 1.14 \times 10^{0} $  & 2045 & 0.5117 & $ 1.14 \times 10^{0} $  & 0.4489 & $ 1.05 \times 10^{0}$ & 4393 & 0.4784 & $ 1.17 \times 10^{0} $ \\
Thoracic 10 & 0.5248 & $ 1.19 \times 10^{0} $  & 4168 & 0.5444 & $ 1.29 \times 10^{0} $  & 0.4353 & $ 1.04 \times 10^{0}$ & 8501 & 0.5008 & $ 1.30 \times 10^{0} $ \\
Thoracic 12 & 0.4496 & $ 9.24 \times 10^{-1} $ & 2347 & 0.4653 & $ 9.69 \times 10^{-1} $ & 0.4137 & $ 8.20 \times 10^{-1}$ & 6900 & 0.4486 & $ 9.79 \times 10^{-1}$ \\
            \bottomrule            
      \end{tabular}
      }
\end{table}

\begin{table}[htbp]
   \caption{Comparison between the numerical results for the Riemannian gradient and the Riemannian conjugate gradient methods.} 
   \label{tab:table_2}
   \center
   \centering
   \resizebox{\textwidth}{!}{
      \begin{tabular}{@{} *{11}{c}}
            \toprule
            & \multicolumn{5}{c}{Riemannian Gradient Method} & \multicolumn{5}{c}{Riemannian Conjugate Gradient Method} \\
            \cmidrule(lr){2-6} \cmidrule(lr){7-11}
            & \multicolumn{3}{c}{Before bij. correction} & \multicolumn{2}{c}{After bij. correction} & \multicolumn{3}{c}{Before bij. correction} & \multicolumn{2}{c}{After bij. correction} \\
            \cmidrule(lr){2-6} \cmidrule(lr){7-11}
Model Name & SD/Mean & $E(f)$ & \#Fs & SD/Mean & $E(f)$ & SD/Mean & $E(f)$ & \#Fs & SD/Mean & $E(f)$ \\
            \cmidrule(r){1-1} \cmidrule(lr){2-6} \cmidrule(lr){7-11}
Knot        & 0.1393 & $ 1.83 \times 10^{-1} $ & 0 & --- & --- & 0.0925 & $ 8.25 \times 10^{-2}$ & 0 & --- & --- \\
Triangle Torus 7/3 & 0.3289 & $ 4.44 \times 10^{-1} $ & 0 & --- & --- & \textbf{0.0659}
 & $ 1.82 \times 10^{-2}$ & 0 & --- & --- \\
Triangle Torus 10/3 & 0.4936 & $ 8.66 \times 10^{-1} $ & 15 & 0.5040 & $ 8.87\times 10^{-1} $ & 0.1574
 & $ 9.80 \times 10^{-2}$ & 165 & 0.2801 & $ 2.58 \times 10^{-1} $ \\
Bob Isotropic & 0.4638 & $ 1.10 \times 10^{0} $ & 104 & 0.5104 & $ 1.29 \times 10^{0} $ & 0.1982 & $ 2.19 \times 10^{-1}$ & 291 & 0.3336 & $ 5.58 \times 10^{-1}$ \\
Square Knot (1, 8, 0) & 0.0588 & $ 7.69 \times 10^{-2} $ & 0 & --- & --- & 0.0418 & $ 4.44 \times 10^{-2}$ & 0 & --- & --- \\
Circle Knot (3, 5/3) & 0.0346 & $ 2.91 \times 10^{-2} $ & 0 & --- & --- & 0.0324 & $ 2.53 \times 10^{-2}$ & 0 & --- & --- \\
Vertebrae   & 0.3918 & $ 8.13 \times 10^{-1} $ & 37 & 0.3939 & $ 8.19 \times 10^{-1} $ & 0.2506 & $ 3.47 \times 10^{-1}$ & 469 & 0.3130 & $ 5.19 \times 10^{-1} $ \\
Kitten      & 0.5320 & $ 1.49 \times 10^{0} $  & 120 & 0.5440 & $ 1.56 \times 10^{0} $ & 0.4064 & $ 8.61 \times 10^{-1}$ & 462 & \textbf{0.4402} & $ 9.84 \times 10^{-1} $ \\
Cogwheel 
& 0.3957 & $ 9.73 \times 10^{-1} $ & 0 & --- & --- & 0.2765 & $ 4.73 \times 10^{-1}$ & 0 & --- & --- \\
Chess Horse & 0.9182 & $ 6.23 \times 10^{0} $  &  18 & 0.9188 & $ 6.24 \times 10^{0} $ & 0.8039 & $ 5.30 \times 10^{0}$  & 276 & 0.8141 & $ 5.35 \times 10^{0}$ \\
Rusted Gear 
& 0.6910 & $ 1.88 \times 10^{0} $ & 0 & --- & --- & 0.5539 & $ 1.29 \times 10^{0}$ & 1180 & 0.6137 & $ 1.48 \times 10^{0}$ \\
Meander Ring 
& 0.4421 & $ 7.45 \times 10^{-1} $ & 0 & --- & --- & 0.4370 & $ 7.27 \times 10^{-1}$ & 0 & --- & --- \\
Lumbar 2    & 0.7534 & $ 1.45 \times 10^{0} $  & 113 & 0.7907 & $ 1.46 \times 10^{0}$  & 0.6466 & $ 1.39 \times 10^{0} $ & 4006 & \textbf{0.5557} & $ 1.44 \times 10^{0}$ \\
Lumbar 4    & 0.5125 & $ 1.17 \times 10^{0} $  & 101 & 0.5190 & $ 1.17 \times 10^{0}$  & 0.4518 & $ 1.08 \times 10^{0} $ & 5173 & \textbf{0.4776} & $ 1.17 \times 10^{0} $ \\
Thoracic 10 & 0.5608 & $ 1.29 \times 10^{0} $  & 260 & 0.5683 & $ 1.30 \times 10^{0}$  & 0.4415 & $ 1.09 \times 10^{0} $ & 10236 & \textbf{0.4968} & $ 1.31 \times 10^{0} $ \\
Thoracic 12 & 0.4652 & $ 9.67 \times 10^{-1}$ & 104 & 0.4713 & $ 9.69 \times 10^{-1}$  & 0.4120 & $ 8.60 \times 10^{-1}$ & 6498 & \textbf{0.4460} & $ 9.80 \times 10^{-1}$ \\
            \bottomrule       
      \end{tabular}
      }
\end{table}

From Tables~\ref{tab:table_1} and~\ref{tab:table_2}, we observe that the projected conjugate gradient and the Riemannian conjugate gradient methods always perform better than the other two methods.
Regardless of the method adopted, some mesh models (Knot, Square Knot (1, 8, 0), and Circle Knot (3, 5/3)) never require a correction to ensure the mapping's bijectivity.
When the bijectivity correction is required, the projected conjugate gradient method gives the best SD/Mean results for seven mesh models (Trianle Torus 10/3, Bob Isotropic, Vertebrae, Cogwheel, Chess Horse, Rusted Gear, and Meander Ring), while the Riemannian conjugate gradient method gives better results in terms of the SD/Mean for five mesh models (Kitten, Lumbar 2, Lumbar 4, Thoracic 10, and Thoracic 12). This is highlighted by the bold text in the SD/Mean columns. The increase in the SD/Mean and energy values after applying the bijectivity correction is due to the unfolding of overlapped triangles.

In general, all four algorithms perform better than the original SEM method; this is partly because the initial map is computed by the original SEM fixed-point method, and this provides a way to improve the initial mapping. We experimented with several different initial mappings, and our numerical results proved to be (highly) dependent on the initial mapping.

Figure~\ref{fig:All_models_and_methods_alphaMax_1_MaxIter_100} reports on the convergence behavior of the objective function $E$. For each benchmark mesh model, we plot all four algorithms' convergence behaviors in the same window.
We observe that the objective function is monotonically decreasing during the whole optimization process.
The projected conjugate gradient method outperforms the other methods in the majority of cases.

\begin{figure}[htbp]
  \centering
  \includegraphics[width=\textwidth]{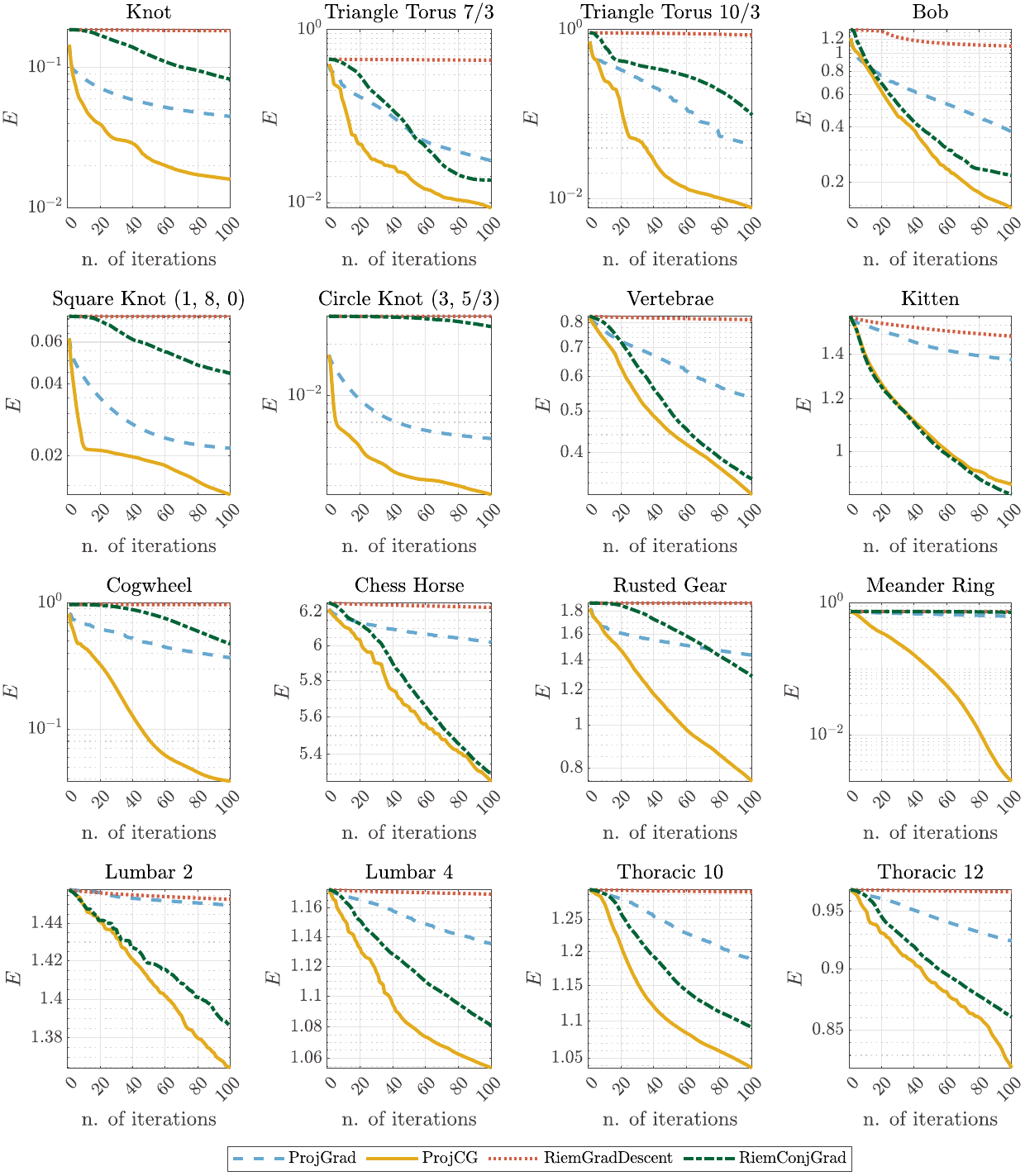}
  \caption{Convergence behavior for the objective function $E$ of the four algorithms for all the benchmark mesh models.}
   \label{fig:All_models_and_methods_alphaMax_1_MaxIter_100}
\end{figure}

\section{Applications} \label{sec:applications}
This section presents applications of area-preserving parameterizations for genus-one surfaces for vertebra registration and texture mapping.

\subsection{Surface registration between two vertebrae}
Given two vertebra surfaces $\mathcal{M}$ and $\mathcal{N}$ with corresponding landmark pairs
$$
\{ (p_\ell, q_\ell) \mid p_\ell \in \mathcal{M},\ q_\ell \in \mathcal{N} \}_{\ell=1}^k,
$$
the goal of surface registration is to find a bijective mapping $\Phi : \mathcal{M} \to \mathcal{N}$ such that $\Phi(p_\ell) = q_\ell$ for every $\ell$. Using toroidal parameterizations, this registration problem can be formulated and solved on a canonical ring-torus domain.

Suppose the toroidal parameterizations $f:\mathcal{M}\to \mathbb{T}^2(R_1,r_1)$ and $g:\mathcal{N}\to \mathbb{T}^2(R_2,r_2)$ of the surfaces $\mathcal{M}$ and $\mathcal{N}$ are computed using Algorithm~\ref{algo:RGD}--\ref{algo:PCG}. To unify the shapes of two tori $\mathbb{T}^2(R_1,r_1)$ and $\mathbb{T}^2(R_2,r_2)$, we first compute the torus coordinates $(\theta,\phi)$ of each point $(x,y,z)\in \mathbb{T}^2$ as
$$
\theta =
\begin{cases}
\arctan\frac{y}{x}, & x>0, \; y \geq 0, \\
2\pi + \arctan\frac{y}{x}, & x>0, \; y<0, \\
\pi + \arctan\frac{y}{x}, & x<0, \\
\frac{\pi}{2}, & x=0, \; y>0,\\
\frac{3\pi}{2}, & x=0, \; y<0,
\end{cases}
\quad\text{and}\quad
\phi =
\begin{cases}
\frac{\pi}{2} + \arcsin\frac{z}{r}, & x^2 + y^2 \geq R^2,\\
\frac{3\pi}{2} - \arcsin\frac{z}{r}, & x^2 + y^2 < R^2.
\end{cases}
$$
We choose $R = \frac{1}{2}(R_1 + R_2)$ and $r = \frac{1}{2}(r_1 + r_2)$. Then, the unified toroidal domain is obtained by the parameterization
\begin{align*}
x(\theta,\phi) &= (R + r\cos\phi) \cos\theta, \\
y(\theta,\phi) &= (R + r\cos\phi) \sin\theta, \\
z(\theta,\phi) &= r \sin\phi.
\end{align*}

The objective function for the surface registration is defined as
$$
E_R(\f) = E(\f) + \lambda \| \f_{\mathtt{P}} - \g_{\mathtt{Q}} \|_{\mathrm{F}}^{2},
$$
where $E(\f)$ is defined in \eqref{eq:objective_function}, and the second term enforces the alignment of the landmark correspondences, with $\f_{\mathtt{P}}$ and $\g_{\mathtt{Q}}$ denoting the values of $\f$ and $\g$ at the prescribed landmark indices $\mathtt{P}$ and $\mathtt{Q}$, respectively. The parameter $\lambda > 0$ is a penalty coefficient that controls the strength of the landmark constraint and is typically set to $0.2$ in our experiments.

The gradient of $E_R$ is given by
$$
\nabla E_R(\f) = \nabla E(\f) + 2 \lambda P^\top(\f_{\mathtt{P}} - \g_{\mathtt{Q}}),
$$
where $\nabla E$ is given in \eqref{eq:grad_E}, and $P\in \{0,1\}^{k\times n}$ is a row-selection matrix given by
$$
P_{\ell, i} =
\begin{cases}
1 & \text{if $i = \mathtt{P}(\ell)$}, \\
0 & \text{otherwise},
\end{cases}
$$
for $\ell = 1,\dots,k$, $i = 1,\dots,n$. 
Here, $\mathtt{P}(\ell)$ denotes the $\ell$th entry of the landmark index set $\mathtt{P}$. 
The minimization of $E_R$ can be analogously carried out using Algorithm~\ref{algo:RGD}--\ref{algo:PCG}.

A landmark-aligned morphing process from $\mathcal{M}$ to $\mathcal{N}$ can be carried out by the linear homotopy $H \colon \mathcal{M} \times [0,1] \to \mathbb{R}^3$ given by
\begin{equation} \label{eq:homotopy}
H(\bv,t) = (1 - t) \, \bv + t \, \Phi(\bv).
\end{equation}
Figure~\ref{fig:morphing} illustrates this morphing process between two surfaces of lumbar vertebrae through four snapshots corresponding to $t=0, \frac{1}{3}, \frac{2}{3}, 1$.

\begin{figure}
    \centering
    \begin{tabular}{cccc}
    \includegraphics[height=3.2cm]{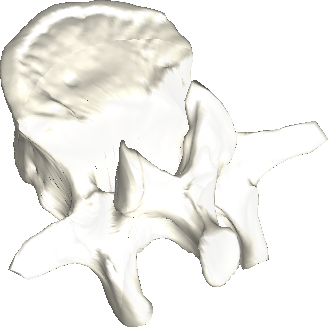} &
    \includegraphics[height=3.2cm]{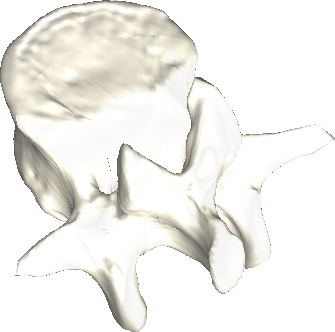} &
    \includegraphics[height=3.2cm]{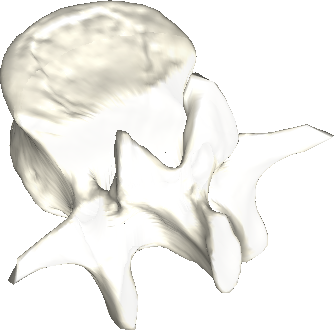} &
    \includegraphics[height=3.2cm]{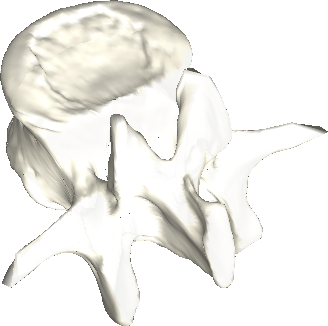} \\
    $t=0$ & $t=\frac{1}{3}$ & $t=\frac{2}{3}$ & $t=1$
    \end{tabular}
    \caption{The morphing process between two lumbar vertebrae.}
    \label{fig:morphing}
\end{figure}

\subsection{Texture mapping}
Texture mapping is a computer graphics technique for applying 2D images to 3D models.
In this section, we illustrate texture mapping through a couple of examples.

The procedure for texture mapping is as follows:
\begin{itemize}
    \item We choose the area-preserving parameterization of a mesh model computed via the projected conjugate gradient method. 
    \item We compute the fundamental domain using Algorithm~\ref{algo:fundamental_domain}.
    \item We use MeshLab~\protect{\cite{MeshLab:2008}} to visualize the models after texture mapping.
    \item Perform scaling and translation.
\end{itemize}

Figures~\ref{fig:texture_mapping_gear} and \ref{fig:texture_mapping_horse} illustrate the texture mapping for two mesh models, ``Rusted Gear'' and ``Chess Horse'', respectively.

\begin{figure}
    \centering
    \subfigure[]{\includegraphics[height=4cm]{Rusted_Gear_v25075.png}}\label{fig:rusted_gear_mesh} \hspace{1.5cm}
    \subfigure[]{\includegraphics[height=4cm]{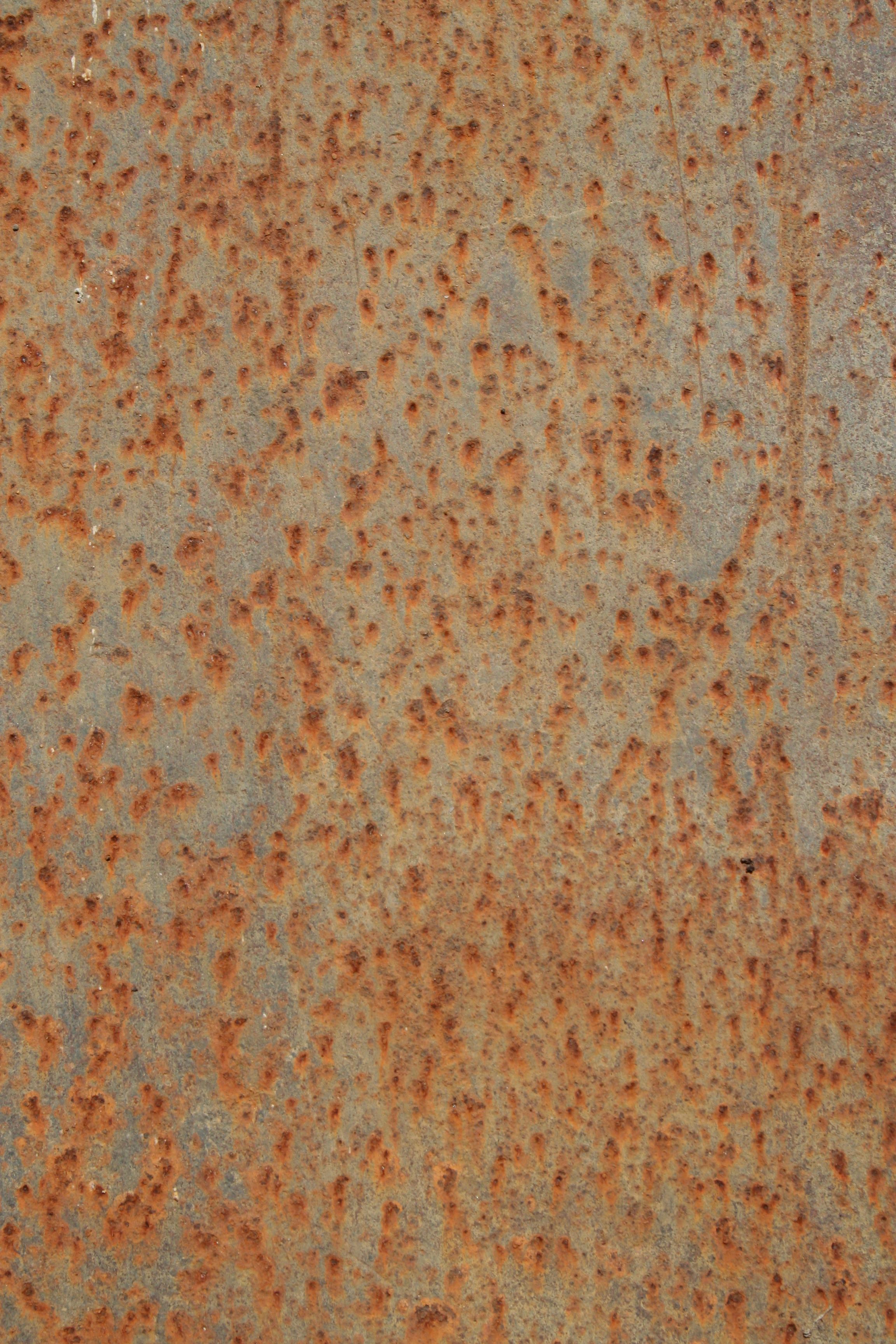}}\label{fig:rusted_gear_texture} \hspace{1.5cm}
    \subfigure[]{\includegraphics[height=4cm]{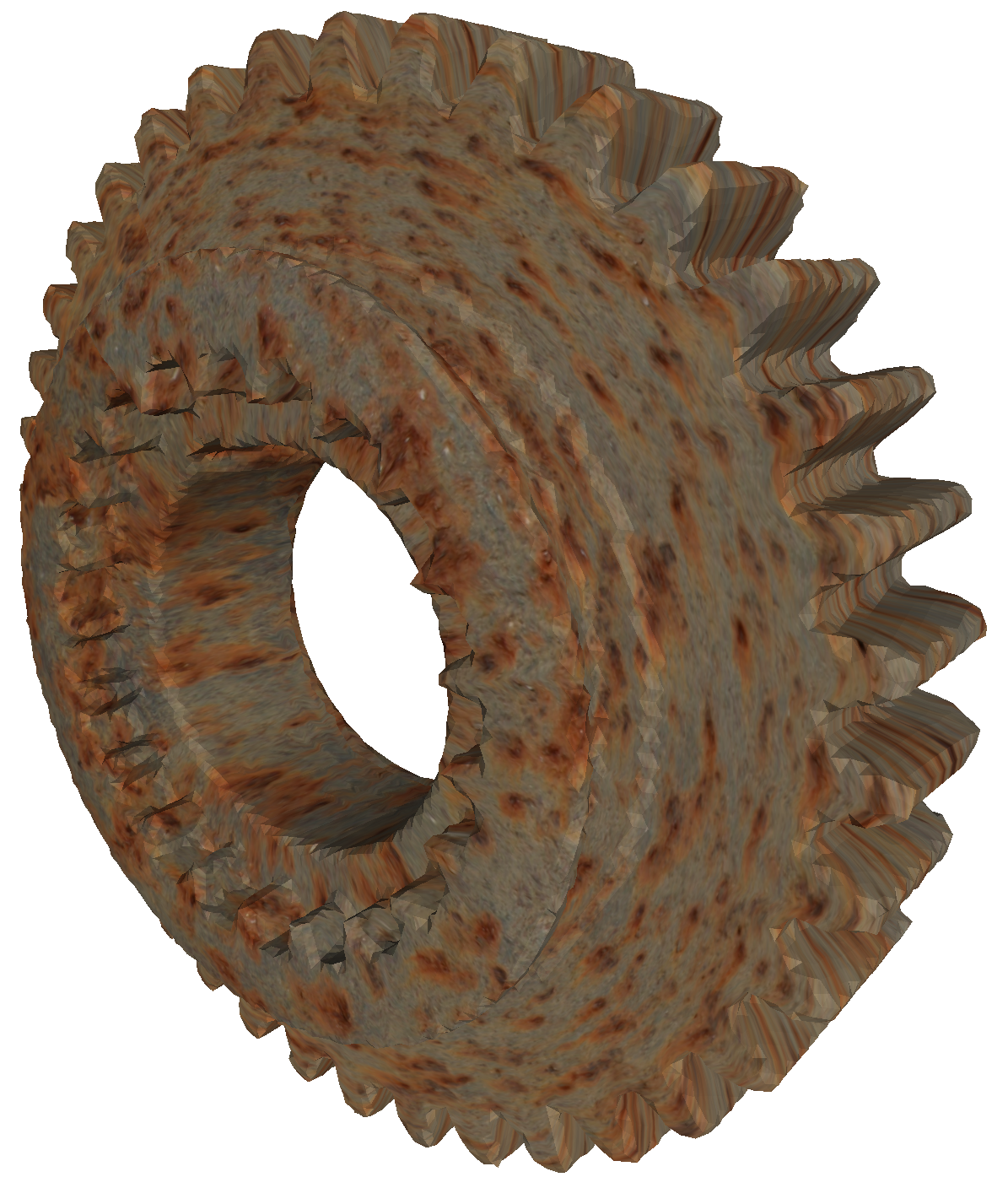}}\label{fig:rusted_gear_w_texture}
    \caption{The ``Rusted Gear'' mesh model, the texture map (b), and the model after texture mapping (c). Both the mesh model and the mapping are from \url{https://www.cgtrader.com/free-3d-models/vehicle/industrial-vehicle/rusted-mechanical-gear}.}
    \label{fig:texture_mapping_gear}
\end{figure}

\begin{figure}
    \centering
    \subfigure[]{\includegraphics[height=4cm]{ChessHorse.png}} \hspace{1.5cm}
    \subfigure[]{\includegraphics[height=4cm]{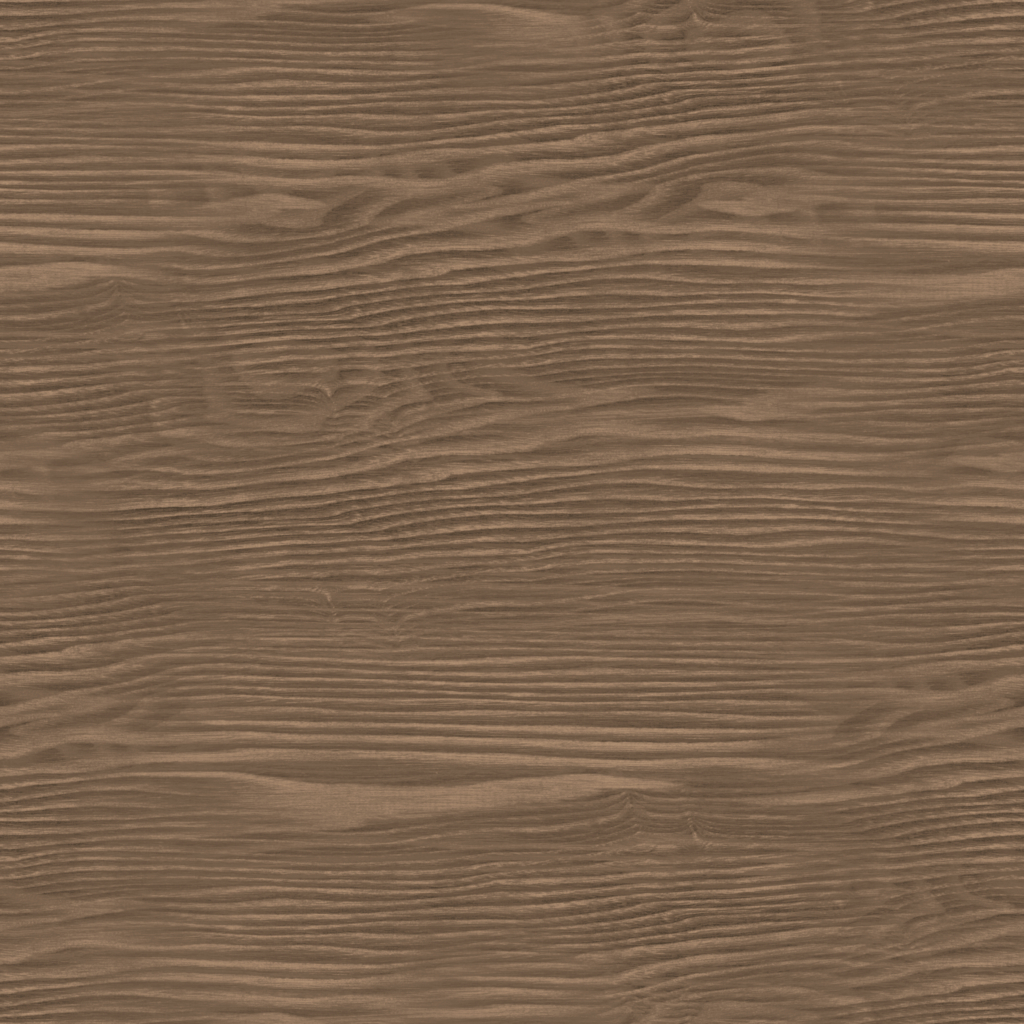}} \hspace{1.5cm}
    \subfigure[]{\includegraphics[height=4cm]{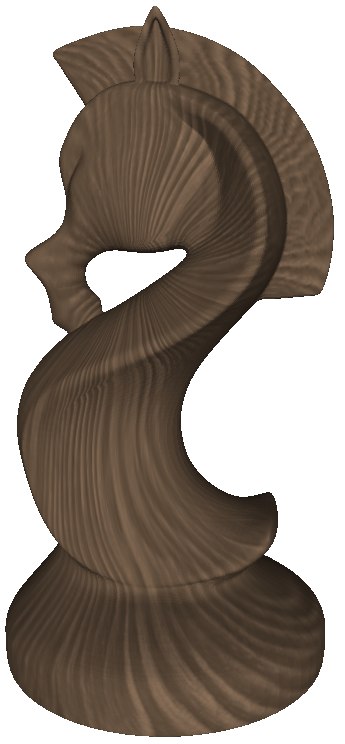}}
    \caption{The ``Chess Horse'' mesh model (a), the texture map (b), and the model after texture mapping (c). The wood texture was taken from \url{https://ambientcg.com/view?id=Wood049}.}
    \label{fig:texture_mapping_horse}
\end{figure}

\section{Concluding remarks and future outlook} \label{sec:conclusions}
We considered the problem of computing toroidal area-preserving mappings of genus-one closed surfaces. We developed the geometry and algorithmic components to propose four algorithms: the projected gradient, the projected conjugate gradient, the Riemannian gradient, and the Riemannian conjugate gradient methods.
Numerical experiments using the four algorithms on several benchmark mesh models demonstrate that the projected conjugate gradient method outperforms the other algorithms in most cases.

\section{Acknowledgments}
The work of the first author was supported by the National Center for Theoretical Sciences, under the NSTC grant 112-2124-M-002-009-. 
The work of the second author was supported by the National Science and Technology Council under Grant 113-2115-M-003-012-MY2, the National Center for Theoretical Sciences, and National Taiwan Normal University through the Higher Education Sprout Project funded by the Ministry of Education, Taiwan.

\appendix

\section{Calculations of the derivative of the retraction} \label{sec:line_search}

This section is similar in spirit to Appendix C in~\protect{\cite{Sutti_Yueh:2024}}. Here, we specialize it with the derivative of the retraction on the torus, $ \Retraction_{\x} $, defined in~\eqref{eq:retraction_on_torus}.

In the line-search procedure, at a given iteration $k$, we need to check that the new step length $ \alpha_{k} $ satisfies the sufficient decrease condition
\[
   \phi(\alpha_{k}) \leq \phi(0) + c_{1} \alpha_{k} \phi'(0),
\]
where $ \phi(\alpha) $ is the real-valued function of the one real variable $\alpha$ defined by $\phi(\alpha) \coloneqq E(\boldsymbol{\psi}(\alpha))$, where $\boldsymbol{\psi}(\alpha) \coloneqq \Retraction_{\x}(\alpha \boldd) = \Pi(\x + \alpha \boldd) $, with the retraction $ \Retraction_{\x} $ as in~\eqref{eq:retraction_on_torus}.
Evaluating the sufficient decrease condition involves the derivative $ \phi'(0) $, which is given by
\[
    \phi'(0) = \left.\left[ \phi'(\alpha) \right]\right\rvert_{\alpha=0} = \left.\left[\trace\!\big(\nabla E(\boldsymbol{\psi}(\alpha))\tr \boldsymbol{\psi}'(\alpha)\big) \right]\right\rvert_{\alpha=0} = \trace\!\big(\nabla E(\x)\tr \boldsymbol{\psi}'(0)\big),
\]
so we need to compute $ \boldsymbol{\psi}'(0) $.
We first compute $ \boldsymbol{\psi}'(\alpha) $ as follows
\begin{align*}
   \boldsymbol{\psi}'(\alpha) &= \frac{\partial}{\partial \alpha} \, \Pi_{\Torus}(\x + \alpha \boldd) \\
   &= \begin{pmatrix}
       \frac{\partial}{\partial \alpha} \left( ( R + r\cos\phi_{\x + \alpha \boldd} ) \cos\theta_{\x + \alpha \boldd} \right) \\[4pt]
       \frac{\partial}{\partial \alpha} \left( ( R + r\cos\phi_{\x + \alpha \boldd} ) \sin\theta_{\x + \alpha \boldd} \right) \\[4pt]
       \frac{\partial}{\partial \alpha} \left( r\sin\phi_{\x + \alpha \boldd} \right)
   \end{pmatrix} \\
   &= \begin{pmatrix}
       \left( R + r \, \cos\phi_{\x + \alpha \boldd} \right) \frac{\partial}{\partial \alpha} \cos\theta_{\x + \alpha \boldd} + r \cos\theta_{\x + \alpha \boldd} \, \frac{\partial}{\partial \alpha} \cos\phi_{\x + \alpha \boldd} \\[4pt]
       \left( R + r \, \cos\phi_{\x + \alpha \boldd} \right) \frac{\partial}{\partial \alpha} \sin\theta_{\x + \alpha \boldd} + r \sin\theta_{\x + \alpha \boldd} \, \frac{\partial}{\partial \alpha} \cos\phi_{\x + \alpha \boldd} \\[4pt]
       r \, \frac{\partial}{\partial \alpha} \sin\phi_{\x + \alpha \boldd}
   \end{pmatrix}.
\end{align*}
Now, we need to compute the derivatives $ \frac{\partial}{\partial \alpha} \cos\theta_{\x + \alpha \boldd} $, $ \frac{\partial}{\partial \alpha} \sin\theta_{\x + \alpha \boldd} $, $ \frac{\partial}{\partial \alpha} \cos\phi_{\x + \alpha \boldd} $, and $ \frac{\partial}{\partial \alpha} \sin\phi_{\x + \alpha \boldd} $, where
\[
   \cos\theta_{\x + \alpha \boldd} = \frac{x_{1} + \alpha d_{1}}{\sqrt{(x_{1} + \alpha d_{1})^{2}+(x_{2} + \alpha d_{2})^{2}}}, \qquad \sin\theta_{\x + \alpha \boldd} = \frac{x_{2} + \alpha d_{2}}{\sqrt{(x_{1} + \alpha d_{1})^{2}+(x_{2} + \alpha d_{2})^{2}}}.
\]
For ease of notation, we introduce the following dummy variables (all depending on $\alpha$)
\[
   A \coloneqq d_{1} (x_{1} + \alpha d_{1}) + d_{2} (x_{2} + \alpha d_{2}), \qquad B \coloneqq \sqrt{(x_{1} + \alpha d_{1})^{2}+(x_{2} + \alpha d_{2})^{2}},
\]
\[
   C \coloneqq B - R, \qquad D \coloneqq x_{3} + \alpha d_{3}, \qquad G \coloneqq \sqrt{C^{2} + D^{2}}, \qquad  H \coloneqq \frac{\frac{A C}{B} + d_{3} D}{G^{3}}.
\]
The derivatives are
\[
   \frac{\partial}{\partial \alpha} \cos\theta_{\x + \alpha \boldd} = \frac{d_{1}}{B} - (x_{1} + \alpha d_{1}) \, \frac{A}{B^{3}},
\qquad
   \frac{\partial}{\partial \alpha} \sin\theta_{\x + \alpha \boldd} = \frac{d_{2}}{B} - (x_{2} + \alpha d_{2}) \, \frac{A}{B^{3}}.
\]
Analogously, for the elevation angle $\phi$, we have
\[
   \cos\phi_{\x + \alpha \boldd} = \frac{\sqrt{(x_{1} + \alpha d_{1})^{2}+(x_{2} + \alpha d_{2})^{2}} - R}{\sqrt{\left(\sqrt{(x_{1} + \alpha d_{1})^{2}+(x_{2} + \alpha d_{2})^{2}} - R \right)^{2} + (x_{3} + \alpha d_{3})^{2} }}
\]
\[
   \sin\phi_{\x + \alpha \boldd} = \frac{x_{3} + \alpha d_{3}}{\sqrt{\left(\sqrt{(x_{1} + \alpha d_{1})^{2}+(x_{2} + \alpha d_{2})^{2}} - R \right)^{2} + (x_{3} + \alpha d_{3})^{2} }}
\]
The derivatives are:
\[
   \frac{\partial}{\partial \alpha} \cos\phi_{\x + \alpha \boldd} = \frac{A}{B G} - C H,
\qquad
   \frac{\partial}{\partial \alpha} \sin\phi_{\x + \alpha \boldd} = \frac{d_{3}}{G} - D H.
\]

\bibliographystyle{aomalpha}

\begin{small}
    \bibliography{Torus_manuscript_arXiv.bib}
\end{small}


\end{sloppypar}
\end{document}